\documentclass[a4paper, 10pt, leqno,final]{article}

\usepackage[T1]{fontenc}	
\usepackage[utf8]{inputenc}
\usepackage[english]{babel}

\usepackage{amsmath}
\usepackage{amssymb}
\usepackage{amsthm}
	\theoremstyle{plain}
		\newtheorem{mainthm}{\textsc{Theorem}}	
\newtheorem{mainprop}{\textsc{Proposition}}	
		\newtheorem{thm}{Theorem}[section]	
			
		\newtheorem{maincor}{\textsc{Corollary}}
		\newtheorem{lem}[thm]{Lemma}		
		\newtheorem{prop}[thm]{Proposition}

	\theoremstyle{definition}
		\newtheorem{defn}[thm]{Definition}	
		\newtheorem{ex}[thm]{Example}		
	\theoremstyle{remark}
		\newtheorem{rem}[thm]{Remark}		
		\newtheorem{note}[thm]{Notation}	

\numberwithin{equation}{section}	

\usepackage{mathtools}	
\usepackage{mathrsfs}	
\usepackage{eucal}	

\usepackage{braket}	
\usepackage[all,pdf]{xy}
\usepackage{mparhack}	
\usepackage{relsize}	

\usepackage{a4wide}

\usepackage{booktabs}	
\usepackage{multirow}	
\usepackage{caption}	
\usepackage{subfig}

\usepackage{varioref}	
\usepackage{footmisc}

\usepackage{graphicx}	
\usepackage{epsfig}
\usepackage{enumerate}	
\usepackage[colorlinks=true, linkcolor=black, citecolor=black,
urlcolor=blue]{hyperref}		
\mathtoolsset{showonlyrefs}

\newcommand{\GL}{\mathrm{GL}}
\newcommand{\cfsa}{\mathcal{CF}^{sa}}
				
\newcommand{\trasp}[1]{{#1}^\mathsf{T}}		
\newcommand{\traspm}[1]{{#1}^\mathsf{-T}}	
\newcommand{\iMor}{\mathrm{n_-}}	
\newcommand{\iMorh}[1]{\mathrm{n}_-^{#1}}		
\newcommand{\R}{\mathbf{R}}
\newcommand{\C}{\mathbf{C}}
\newcommand{\Q}{\mathbf{Q}}
\newcommand{\U}{\mathbf{U}}

\newcommand{\OO}{\mathrm{O}}
\newcommand{\TT}{\mathbf{T}}	
\newcommand{\HH}{\mathbf{H}}	
\newcommand{\mmod}{\quad({\mathrm{mod}\ 2)}}
\newcommand{\irel}{I}

\newcommand{\Sp}{\mathrm{Sp}}

\newcommand{\Lagr}{\Lambda}
	
\newcommand{\Real}{\mathrm{Re}}
\newcommand{\Fix}{\mathrm{Fix}}
\newcommand{\Lin}{\mathscr{L}}
\newcommand{\Graph}{\mathrm{Gr\,}}

\newcommand{\Mat}{\mathrm{Mat}}

\newcommand{\Hess}{D^2}
\newcommand{\fredsa}{\mathcal F^{\textup{sa}}}
\newcommand{\Bsym}{\mathrm{B}_{\textup{sym}}}
\newcommand{\Gr}{\mathrm{Gr}}

\newcommand{\im}{\mathrm{rge}\,}

\newcommand{\die}[1]{ D_{#1}}	
\newcommand{\cic}[1]{ C_{#1}}

\newcommand{\norm}[1]{\left\| #1 \right\|}
\newcommand{\N}{\mathbb{N}}		
\newcommand{\iMas}{\iota_{\scriptscriptstyle{\mathrm{geo}}}}
\newcommand{\iRel}{\iota_{\scriptscriptstyle{\mathrm{spec}}}}
\newcommand{\iCLM}{\mu^{\scriptscriptstyle{\mathrm{CLM}}}}
\newcommand{\Z}{\mathbf{Z}}		
\newcommand{\I}{\mathbb{I}}		
\providecommand{\myfloor}[1]{\left \lfloor #1 \right \rfloor }
\newcommand{\coindex}{\mathrm{n_+}}
\newcommand{\iindex}[1]{\mu_{\scriptscriptstyle{\mathrm{Mor}}}\left[#1\right]}
\newcommand{\coiindex}[1]{\mathrm{n_+}\left[#1\right]}

\newcommand{\ssgn}[1]{\sgn\left[#1\right]}
\newcommand{\noo}[1]{\overset {\mbox{%
\lower1pt\hbox{${\scriptscriptstyle o}$}}}n^{\mbox{%
\lower2pt\hbox{$\scriptscriptstyle #1$}}}}

\DeclareMathOperator{\diag}{diag}
\DeclareMathOperator{\spfl}{sf}	
\DeclareMathOperator{\sgn}{sgn}		
	
\DeclareMathOperator{\rk}{rank}	
\DeclareMathOperator{\ord}{ord}		

\renewcommand{\leq}{\leqslant}
\renewcommand{\geq}{\geqslant}

\renewcommand{\=}{\coloneqq}

\newcommand{\email}[1]{\href{mailto:#1}{\textsf{#1}}}

\newcommand{\Id}{I}
	
\title{A dihedral Bott-type iteration formula and  stability of symmetric
periodic orbits}
\author{Xijun Hu\thanks{The author is partially supported by NSFC (No.11425105).},
Alessandro Portaluri
\thanks{The
author is partially supported by the project ERC Advanced Grant 2013
No.~339958 ``Complex Patterns for Strongly Interacting Dynamical Systems ---
COMPAT”, by Prin 2015 ``Variational methods, with applications to
problems in mathematical physics and geometry” No.~2015KB9WPT$\_$001
and
by Ricerca locale 2015  ``Semi-classical trace formulas and their application
in physical chemistry” No.~Borr$\_$Rilo$\_$16$\_$01 .} , Ran Yang }
\date{\today}

\date{\today}
\begin{document}
 \maketitle

\begin{abstract}
In 1956, Bott in his celebrated paper on closed geodesics and Sturm intersection
theory, proved  an Index Iteration Formula for closed
geodesics on Riemannian manifolds. Some years later, Ekeland improved this
formula  in the case of convex Hamiltonians and,  in 1999, Long
generalized the Bott iteration formula by putting in its natural symplectic context
and  constructing a very effective Index Theory. The literature about this
formula is quite  broad and the dynamical implications in the Hamiltonian world
(e.g. existence, multiplicity, linear stability etc.)   are enormous.

Motivated by the  recent discoveries on the stability properties of symmetric periodic
solutions of singular Lagrangian systems,  we establish a Bott-type
iteration formula for  dihedrally equivariant Lagrangian and   Hamiltonian systems.

We finally apply our theory for computing the  Morse indices   of the
celebrated Chenciner and Montgomery figure-eight orbit for the planar three body problem in different 
equivariant spaces. Our last
dynamical consequence is an  hyperbolicity criterion
for reversible Lagrangian systems.
\vskip0.2truecm
\noindent
\textbf{AMS Subject Classification:} 70F10, 37C80, 53D12, 58J30.
\vskip0.1truecm
\noindent
\textbf{Keywords:} Maslov index, Spectral Flow, Equivariant dihedral group
action, Bott iteration
formula, linear stability, 3-body problem.
\end{abstract}

\tableofcontents

\section{Introduction and description of the main results}\label{sec:intro}

Symmetric periodic orbits of Lagrangian or more general Hamiltonian systems
have been discovered in the
last decades by developing suitable variational methods on a space of loops symmetric
 with respect to a chosen
symmetry compact  Lie group $G$.  The well-known {\em Palais principle of symmetric criticality\/}
(cf. \cite{Pal79}) claims that
critical points of the restriction of a $G$-invariant  functional to the space fixed  by $G$
are critical points of the functional.  Variational minimization
methods  on the space of $G$-equivariant loops, were recently successfully employed by many authors for
proving  the existence of amazing symmetric periodic orbits; e.g. the Chenciner-Montgomery
figure-eight solution of the planar three-body problem with
equal masses. (Cf., for further details, \cite{CM00, Che02a, Che02b, FT04} and references
therein). In the aforementioned problem the idea of minimizing the Lagrangian action on a space of loops
symmetric with respect to a chosen symmetry group plays a crucial role in order to prove the
existence of  collisionless periodic orbits.

The existence of $G$-symmetric periodic orbits of a Hamiltonian system is the first step in order to
penetrate the intricate dynamics of the problem. The second step in this analysis is to investigate the
stability and  multiplicity properties of these solutions. Many techniques have been
developed in the last years for tackling both these problems and among all of them a central role is
essentially played by the index theory. An enormous contribution has been given in the last decades
by Long and his collaborators in the construction of an index theory based on a
symplectic invariant nowadays known in literature as {\em Maslov  index.\/} A central device in
order to investigate the aforementioned problems is based on the celebrated {\em Bott-type
iteration formula\/} for the Maslov-type index, which directly provides an estimate of the
elliptic height (i.e. the total algebraic multiplicity of all eigenvalues on the unit circle of the complex
plane).

Bott-type iteration formula represents the starting point of our analysis and
the initial motivations for  further  investigations. This formula was introduced by author in \cite{Bot56}
for his investigation on closed geodesics on Riemannian manifolds and some years later, Ekeland in \cite{Eke90}
and Long in \cite{Lon02}  put it into a symplectic context by generalizing such a formula to any linear $T$-periodic
Hamiltonian system.  The idea behind  this formula is based on the fact that if $\gamma$ denotes the
fundamental solution of a linear Hamiltonian $T$-periodic system, we can associate
an integer, let's say $\iota_1(\gamma;[0,T])$. It is very natural to single out for considerations also the
indices  $\iota_1(\gamma;[0,2\,T])$, $\iota_1(\gamma;[0,3\,T])$, etc. corresponding to the intervals which are
multiples of the basic period $T$. Bott's iteration formula, essentially allow us to relate the index
$\iota_1(\gamma;[0,k\,T])$ to $\iota_1(\gamma;[0,T])$ and to the Floquet multipliers of the Hamiltonian system.

Let $z$ be a periodic solution of the Hamiltonian system corresponding to the Hamiltonian function $H$ of class $\mathscr C^2$; namely
\begin{equation}\label{eq:hamsys-intro}
\begin{cases}
\dot z(t)= J\nabla H\big(t,z(t)\big),\qquad t \in [0,T]\\
z(0)=z(T)
\end{cases}
\end{equation}
with $J\= \begin{bmatrix}  0 & -\Id\\ \Id & 0 \end{bmatrix}$ and where $\Id$ denotes the identity on $\R^m$. Its associated
fundamental solution $\gamma$ is the matrix-valued solution of the linear initial value problem obtained by
linearizing  the Equation \eqref{eq:hamsys-intro} along $z$; i.e.
\begin{equation}
\begin{cases}
\dot \gamma(t)=J D^2 H\big(t, z(t)\big)\gamma(t)\\
\gamma(0)=\Id.
\end{cases}
\end{equation}
It is well-known that $\gamma$ is a path in the symplectic group of  $(\R^{2m}, \omega)$ where $\omega$ denotes the
standard symplectic structure. If $\U$ is the unit circle of the complex plane, we denote by $\iota_\omega(\gamma)$ the so-called
{\em $\omega$-Maslov-type index of $\gamma$\/} defined as intersection index of $\gamma$ with a transversally oriented 1-codimensional
manifold in the symplectic group $\Sp(2m)$. (We refer the interested reader to \cite{Lon02} and references therein
for further details).

Based on the index function $\iota_\omega$, Long established a Bott-type iteration formula to any continuous path 
$\gamma:[0,T] \to \Sp(2m)$such that $\gamma(0)=\Id$.  More precisely, for any $T>0$, $\gamma$ as before, $z \in \U$ and $n \in \N$,
\begin{equation}\label{eq:classical-Bott}
\begin{split}
\iota_z(\gamma, n)=\sum_{\omega^n=z} \iota_\omega(\gamma)&\\
\nu_z(\gamma, n)&=\sum_{\omega^n=z} \nu_\omega(\gamma)
\end{split}
\end{equation}
where $\nu_\omega(\gamma)$ denotes the $\omega$-nullity of $\gamma$ defined as follows
\[
\nu_\omega(\gamma)\= \dim_\C\ker_\C\big(\gamma(T)-\omega\Id).
\]
Motivated by some concrete problems in Celestial Mechanics, some years ago, in 2009 in fact,
authors in \cite[Theorem 1.1]{HS09} generalized the Bott-type iteration formula given in Equation \eqref{eq:classical-Bott}
in a form which is crucial for applying it, if we are in presence of a cyclic group  $\Z_{n}$ acting on the
orbit. The main idea in order to establish a $\Z_n$-invariant  Bott-type iteration formula, is essentially based on a $\Z_n$-equivariant
decomposition of the path space into isotypical components.

Few years ago, more precisely in 2006, authors in \cite{LZZ06}, working on the  famous Seifert conjecture (claiming the existence of 
 at least $n$ geometrically distinct brake orbits on a given regular compact hypersurface $\Sigma$),  were able to prove 
  the existence of two brake orbits on $\Sigma$ under an additional condition that $H$ is even. The proof of this result is essentially based on  
  a new Maslov-type index theory for brake orbits and Morse
theory applied to the Hamiltonian action functional. Some related formulas, although in a quite different context, were recently
found by Liu and Zhang in \cite{LZ14a} and \cite{LZ14b}. In this paper, authors
established some  Bott-type iteration formulas for the  $L$-index (namely the Maslov-type index of symplectic paths associated
with a Lagrangian subspace). It is worth to quote the paper \cite{LT15} where authors, proved a Bott-type iteration
formula for the Maslov $P$-index which can be taken as a generalization of the standard Bott-type iteration formula where $P=I_{2n}$. Recently, the authors in
\cite{WZ16} gave a direct proof of the iteration formulae for the Maslov-type indices of symplectic paths by using a splitting of the nullity to yield a splitting formula for the Maslov-type indices of symplectic paths in weak symplectic Hilbert space.

Pushing further this analysis and motivated by the fact that many interesting  periodic orbits are actually symmetric with
respect to a more general group action (like the aforementioned figure-eight orbit which is $\die{6}$-symmetric
where $\die{n}$ denotes the dihedral group of order $2n$), in this paper we establish an abstract $\die{n}$-equivariant spectral
 flow formula and a new $\die{n}$-equivariant Bott-type iteration formulas for Hamiltonian (resp. for Lagrangian)
 systems,  through a suitable isotypical decomposition of the path space of the phase (resp. configuration) space.

 \subsection{Equivariant set-up, description of the problem and main results}\label{subsec:equivariant-setup}

 Let $E$ be a complex separable Hilbert space and  $G$ be a finite (abstract) group acting on $E$. Denoting by
 $\U(E)$ the group of {\em unitary operators\/}, we assume that $G$ acts on $E$ through its {\em unitary
 representation\/}
 \[
  G \ni g \longmapsto U_g \in \U(E) .
\]
 We start to consider the (finite)  cyclic group of order $n$ presented as follows
$\cic{n}\=\langle r| r^n=1\rangle$, where we denoted by $e$ the identity of the group. It is 
well-known that it  can be represented by the group
of rotations through angles $2k\pi/n$ around an axis. We denote by
$\die{n}$ the semi-direct product of $\cic{n}$ and $\cic{2}$; in
symbols $\die{n}= \cic{n} \rtimes \cic{2}$,
where $\cic{n}\=\langle r| r^n=e\rangle$ and
$\cic{2}\=\langle s| s^2=e\rangle$.\footnote{%
We observe that
$\die{1}$ and $\die{2}$ are atypical groups and more precisely $\die{1}$ is the
cyclic group of order $2$ whilst $\die{2}$ is the {\em Klein four-group\/}. It
is worth
noticing that for $n \geq 3$, $\die{n}$ is non-abelian.}
The group $\die{n}$ is usually termed  {\em dihedral  group\/} of degree $n$ and
order $2n$. For $n \geq 3$, $\die{n}$ is  the
group of symmetries of a regular $n$-gon in the plane, namely, the group of all
the
plane symmetries that preserves a regular $n$-gon. It contains $n$ rotations
(which form a subgroup isomorphic to $\cic{n}$) as well as  $n$ reflections.
From an algebraic viewpoint, we point out that,  it is a metabelian
group having the cyclic normal
subgroup $\cic{n}$ of index $2$ and the following presentation
\begin{equation}\label{eq:diedrale-1}
 \die{n}\= \langle r,s|r^n=s^2=1,\  srs=r^{-1}\rangle.
\end{equation}
Each element of $\die{n}$ can  be uniquely written, either in the form $r^k$,
with $0 \leq k \leq n-1$ (if it belongs to $\cic{n}$), or in the form $sr^k$,
with $0 \leq k \leq n-1$  (if
it doesn't belong to $\cic{n}$). Observe that the relation $srs=r^{-1}$ implies
that $sr^ks=r^{-k}$ and $(sr^k)^2=1$. It is worth noticing also that the irreducible
representations of $\die{n}$ are one or two dimensional. More precisely, if $n$ is
odd except the (trivial) identity and the minus identity representations, all others are two-dimensional;
if $n$ is even we have other than the previous one-dimensional irreducible
representations also the sign representations whilst all other are
two-dimensional.

With a slight abuse of notation, we denote with the same symbol the dihedral
group as well as its unitary representation in $E$. Thus, the group $\die{n}$
(actually  a unitary  representation of the dihedral group
$\die{n}$) is presented as follows
\begin{equation}\label{eq:image-diedrale-unitaria-app-1}
 D_n \= \langle \mathcal R,  \mathcal  S \in \U(E)| \mathcal R^n=
  \mathcal S^2=( \mathcal  S  \mathcal  R)^2= \Id_E\rangle\subset \U(E).
\end{equation}
As direct consequence of the
spectral mapping theorem, it readily follows that the spectrum of
$ \mathcal  R$ is given by $\sigma( \mathcal  R)=\Set{\zeta_n^k\in \C|k=0, \dots,
n-1}.$ Furthermore, if  $E_k\=\ker(  \mathcal R-\zeta_n^k\Id_E)$
 denotes the spectral space corresponding to
the eigenvalue $\zeta_n^k$, by the spectral theory of  normal operators, it is
well-known that $E_k$ are mutually orthogonal.

We define the following closed subspaces of $E$
\begin{equation}\label{eq:Fk-intro}
F_{k}\= \begin{cases}
E_0  & \textrm{ if } k=0\\
E_k \oplus E_{-k} & \textrm{ if } k=1,\cdots, \myfloor{(n-1)/2}\\
\begin{cases}
E_{n/2}  & \textrm{ if } n \textrm{ is even }\\
0 & \textrm{ otherwise}
\end{cases} & \textrm{ if } k=n/2
\end{cases}
\end{equation}
where we denoted by $\myfloor{\cdot}$ the integer part.
(We refer the interested reader to  Appendix \ref{subsec:group-algebras}, for further details).  For any $k$, he subspace
 $F_k$ defined in Equation \eqref{eq:Fk-intro} is a
$\die{n}$-module with the action given by
\begin{multline}
 \cic{n} \times F_k \to F_k:\big(\zeta_n^k,v\big)\longmapsto
\mathcal  R\, v\=\begin{bmatrix}
\zeta_n^{k}\, \Id_E&0\\0 & \zeta_n^{-k}\,\Id_E
\end{bmatrix}\,v\quad \textrm { and }	\\
 \cic{2} \times F_k \to F_k:\big(s,v\big)\longmapsto
  \mathcal S\, v\=\begin{bmatrix}
0&\Id_E\\ \Id_E &0
\end{bmatrix}\,v.
\end{multline}
Thus, we get a decomposition of the Hilbert space $E$ into mutually orthogonal
$\die{n}$-stable modules, given by
\begin{equation}\label{eq:decomposition-2-intro}
 E=F_0\oplus \dots \oplus F_{\bar n}
\end{equation}
where $\bar n $ denotes the largest integer not greater than $\myfloor{n/2}$.
\begin{rem}
We observe that the decomposition given in Equation
\eqref{eq:decomposition-2-intro} is the isotypic decomposition of $E$  induced
by the unitary representation of the dihedral group. In particular each subspace $E_k$ is
given by the direct sum of (infinitely many) one-dimensional irreducible
representations of the cyclic group. (Cf. Appendix \ref{subsec:group-algebras} and references therein).
\end{rem}
We denote by $\fredsa(E)$ the space of all bounded selfadjoint Fredholm operators on $E$ and
let  $\mathcal  A: [0,1] \to \fredsa(E)$ be a continuous
path of  operators commuting with $ \mathcal R$ and $\mathcal
S$; namely
\[
 \mathcal A(\lambda)\,{\mathcal R}= {\mathcal R}\, \mathcal A(\lambda) \textrm{ and }
   \mathcal A(\lambda)\, \mathcal S=  \mathcal S\, \mathcal A(\lambda) \textrm{ for all }
\lambda \in [0,1].
\]
For each $k =0, \dots, \bar n$, let  $ \mathcal A_k$ be the continuous path
defined by
$ \mathcal  A_k\= \mathcal  A\big\vert_{ F_k}:[0,1] \to  \fredsa(
F_k)$.  With respect to this orthogonal decomposition, the path $\mathcal A$ can be
written as 
\[
 \mathcal A(\lambda)= \mathcal A_0(\lambda) \oplus \dots \oplus  \mathcal A_{\bar n}(\lambda) \textrm{ for }\lambda\in [0,1].
 \]
As  direct consequence of the decomposition into (mutually
orthogonal)$\die{n}$-stable modules given in Formula \eqref{eq:decomposition-2-intro} as well as
the additivity property of the spectral flow under direct sum, the following  result holds.
\begin{mainprop}\label{thm:Bott-diedrale}({\bf An equivariant spectral flow
formula\/})
Under the previous notation, we have
\begin{equation}\label{eq:somma-spfl-2}
 \spfl( \mathcal A;  [0,1])= \spfl(\mathcal  A_0; [0,1]) + \dots + \spfl(\mathcal
A_{\bar n}; [0,1]).
  \end{equation}
\end{mainprop}
\begin{rem}
It is worth to observe that all the spectral flow formulas can be established for more general class of paths
of selfadjoint Fredholm operators (e.g. for gap continuous path of selfadjoint Fredholm operators).
\end{rem}
For each $k=0, \dots, \bar n$ and $h=0, \dots, n-1$, we define the
closed subspaces
\begin{equation}\label{eq:i0}
F_{k,h}^\pm=\Set{u \in F_k| \mathcal S \mathcal R^h\, u = \pm \, u} \quad \textrm{ and }\quad 
 F_{h}^\pm=\bigoplus_{k=1}^{\bar n} F_{k,h}^\pm
\end{equation}
Since $ \mathcal R\big |_{E_0}=\Id$,   then $F_{0,h}^\pm=F_{0,0}^\pm=\Set{u \in F_0|
 \mathcal S\, u = \pm \, u} .   $
If  $n$ is even, $ \mathcal R\big|_{E_{n/2}}=-\Id$, so we have
$F_{n/2,h}^\pm=\Set{u \in F_{n/2}| (-1)^h \mathcal S\, u = \pm \, u} . $
It is worth to note  that, for $k=1, \dots, \myfloor{(n-1)/2}$, we get
\[
F_{k,h}^\pm= \Set{\begin{bmatrix} I\\ \mathcal S
\mathcal  R^h\end{bmatrix}\, z| z \in E_k}
\]
and by a straightforward calculation for any $h=0, \dots, n-1$, also that  
\[
\big( \mathcal S\mathcal R^{h}\big) \mathcal K_\lambda= \mathcal K_\lambda
\textrm{ where }
\mathcal K_\lambda \= \ker \mathcal A_\lambda\qquad \textrm{ and }
 \lambda \in [0,1].
\]
In the case of $\die{n}$-equivariant path of bounded selfadjoint Fredholm operators, the parity
of the spectral flow actually depends only on the restriction of the path on the finite dimensional subspace
$E_0$  (resp. $E_0, E_{n/2}$) if $n$ is odd (resp. $n$ is even).
\begin{mainprop}\label{thm:parity-equi}
Let $\mathcal A \in \mathscr C^0\big([0,1]; \fredsa(E)\big)$ a $\die{n}$-equivariant path. If
\begin{itemize}
\item  $n$ is odd, then
\[
\spfl( \mathcal A;[0,1])\equiv\spfl\left( \mathcal A\big|_{E_0};[0,1]\right) \mmod
 \]
\item $n$ is even, then
\[
\spfl(\mathcal A;[0,1])\equiv \spfl\left(\mathcal A\big |_{E_0};[0,1]\right)+\spfl\left(\mathcal
A\big|_{E_{n/2}};[0,1]\right)\mmod .
\]
\end{itemize}
\end{mainprop}
Let $h \in \{1,\dots, n-1\}$  and let  us consider a subgroup $G_h$ in $D_n$ generated
by  $\mathcal S\mathcal R^h$; thus
\[
G_h\= \langle \mathcal S \mathcal R^h| \mathcal R, \mathcal S \textrm{ are the generators of } D_n\rangle.
\]
We let
 $E^+_h$ and  $E^-_h$ be the invariant
subspaces respectively given by  $E_h^+\=\Set{x \in E| x=\mathcal S \mathcal  R^h x }$ and
$E_h^-\=\Set{x \in E|x=-\mathcal  S \mathcal R^h x }$.  It is easy to check that
 \[
 E^\pm_h=\bigoplus_{k=0}^{\bar n} F_{k,h}^\pm.
 \]
 As direct consequence of the additivity properties of the spectral flow, the following spectral flow summation
 formulas hold.
\begin{mainprop}\label{thm:cor2.7}
Let $\mathcal A \in \mathscr C^0\big([0,1]; \fredsa(E) \big)$ pointwise commuting with the generators
$\mathcal R$ and $\mathcal S$ of the unitary representation of $\die{n}$  in $E$
and let  $h=0, \dots, n-1$. Thus we have
 \[
 \spfl(\mathcal  A;[0,1])=\spfl(\mathcal A\big|_{E^+_h};[0,1])+\spfl(
\mathcal A\big|_{E^-_h};[0,1])
\]
 \[
 \spfl(\mathcal A\big |_{E^\pm_h};[0,1])=\sum^{\bar n}_{k=0} \spfl\left(
\mathcal A\big |_{F_{k,h}^\pm}; [0,1]\right) .\]
\end{mainprop}
An interesting application of the spectral formula proved in Proposition \ref{thm:cor2.7} is the following.
We assume that $\widetilde {\mathcal A}$ is a selfadjoint essentially positive Fredholm operator; thus, in particular,
$\widetilde {\mathcal A}$ has a finite   Morse index and let us consider the (analytic) path
$\lambda \mapsto \widetilde{\mathcal A}_\lambda\=\mathcal A+\lambda \mathcal K $ where  $\mathcal K$ denotes a compact and  symmetric
 linear bounded   operator.
Denoting by  $\iMor(\widetilde {\mathcal A})$ the Morse index of $\widetilde {\mathcal A}$ (namely the dimension of the negative
spectral space of $\widetilde {\mathcal A}$), it is well-known that
\[
\iMor(\widetilde {\mathcal A})=\spfl(\widetilde {\mathcal A}_\lambda; \lambda\in[0,1]).
\]
(Cf. Appendix \ref{subsec:spectral-flow} for further details).
As direct consequence of Proposition \ref{thm:cor2.7}, the following result holds.
\begin{maincor}\label{prop2.8}
We assume that the path $\mathcal A$ is $\die{n}$-equivariant.  Then 
 \[
 \iMor(\widetilde {\mathcal A})= \iMor\big(\widetilde {\mathcal A}\big|_{E_{h}^+})+\iMor(\widetilde {\mathcal A} |_{E_{h}^-}\big) ,
 \]
 and
\[
\iMor\big(\widetilde {\mathcal A}\big|_{E^\pm_h}\big)=\sum^{\bar n}_{k=0}
\iMor\big(\widetilde {\mathcal A}\big|_{F_{k,h}^\pm}\big) .
\]
\end{maincor}
Given  $T>0$ we denote by  $T\Z$  the lattice generated by $T \in \R$ and we set  $\TT \=\R/(T\Z)\subset \R^2$
be a circle in $\R^2$ of length $T=|\TT|$. Given
$Q \in  \U(2m)$, let $E$ be denote the $W^{1,2}$ closure of the
set of smooth maps $z:\TT \to \C^{2m}$ such that $z(t)=Q z(t+T)$; thus
\begin{equation}\label{eq:Q-loopspace-intro}
E\=\Set{z \in W^{1,2}(\TT;\C^{2m})| z(t)= Q \, z(t+T)}.
\end{equation}
Let  $M \in  \U(2m)$  and $N \in \U(2m)$ be such that the following commutativity
properties holds
 \begin{equation} \label{eq:relazioni-intro}
M^n=Q,\ \   N^2 = \Id,  \ \ N \, M^*= M\, N.
         \end{equation}
For any $k=0, \dots, n-1$, we denote by $E_k$   the closed  subspace of $E$  given by
\[
E_k= \Set{z \in E|M\, z\left(t+\dfrac{T}{n}\right)
=\zeta_n^k\, z(t)}.
\]
and for $k=1,\dots,\myfloor{(n-1)/2}$, we define   $\widehat M_k, \widehat N_k$ the
following block diagonal matrices
\begin{equation}\label{eq:matrici-cappuccio-intro}
 \widehat M_k\= \begin{bmatrix}
              \zeta_n^{-k} M & 0\\
              0 & \zeta_n^k M
             \end{bmatrix}\quad  \textrm{ and } \quad
 \widehat N_k\= \begin{bmatrix}
              0 & \zeta_n^{-k} \, N\\
              \zeta_n^{k} \, N & 0
             \end{bmatrix}.
\end{equation}
As above, we set $F_0=E_0$,  $F_{n/2}=E_{n/2}$ if $n$ is even,$ F_{n/2}=0$ if $n$ is odd and
finally  $F_k\= E_k \oplus E_{-k}$
for $k=1,\cdots,\myfloor{(n-1)/2}$.  Thus, we have
\begin{equation}\label{eq:F_k-intro}
\begin{split}
F_k=\Set{u \in W^{1,2}(\dfrac{\TT}{n}; \C^{2m} \oplus \C^{2m})|\ \
u\=\begin{bmatrix}
                                                            z\\ w
                                                           \end{bmatrix}
                                                           z \in E_k,\ \  w \in
E_{-k}
                                                          \textrm { and }  u(0)=
      \widehat M_k\, u\left(\dfrac{T}{n}\right)}
      \end{split}.
\end{equation}
We now define the two unitary operators on $E$ as follows
\begin{multline}\label{eq:action-1-intro}
\mathcal M: E \ni  z\longmapsto  (\mathcal M\, z)(\cdot) \=M\, z\left(\cdot+ T/n\right)\in E
\textrm { and }\\
\mathcal N: E \ni  z\longmapsto  (\mathcal N\, z)(\cdot) \=N\, z\left(T/n-t\right)\in E.
\end{multline}
By using the properties given in Equation \eqref{eq:relazioni-intro}, it is immediate
to check that $E$ equipped by the action defined in Formula \eqref{eq:action-1-intro}
turns out a $\die{n}$-equivariant space. Moreover, as direct consequence of the relations on the
generators of the dihedral group $\die{n}$, we also get that, for every $h =0, \dots, n-1$,
$(\mathcal N\,\mathcal M^h)^2=\Id$.
In this way we decompose  $E$ defined in Equation \eqref{eq:Q-loopspace-intro}
into a direct sum of $\die{n}$-closed stable subspaces; namely
$E= F_0 \oplus \cdots \oplus F_{\bar n}$ where $F_k$ were defined in Formula
\eqref{eq:F_k-intro}.

Now, let $Q \in \Sp(2m, \R) \cap \OO(2m)$ be such that
\[
 MJ=JM, \quad M^n=Q, \quad N^2 =\Id, NJ=-JN \textrm{ and finally } N=\trasp{N}
\]
and we assume that  $H \in \mathscr C^2(\R\times \R^{2n},\R)$ is a  Hamiltonian function such
that
\begin{equation}\label{eq:simmetrie-H}
 H(t-T/n, M \,z)=H(t,z) \textrm{ and } H(T-t, N\, z)=H(t,z),\footnote{%
 In the autonomous case we assume that
\[ H(M \,z)=H(z) \textrm{ and } H( N\, z)=H(z).\]}
\end{equation}
Given a  Lagrangian subspace
$L $, we
denote by $z$ a  solution of the Hamiltonian system
\begin{equation}\label{eq:Ham-sys-intro}
\begin{cases}
 z'(t)= \nabla H\big(t, z(t)\big), \quad t \in [0,T]\\
 \big(z(0),z(T)\big) \in L.
\end{cases}
\end{equation}
To any solution $z$, we associate   $\gamma:[0,T] \to \Sp(2m)$ which is
 the fundamental solution of the linearized Hamiltonian system along it
 \begin{equation}\label{eq:i1}
 \begin{cases}
  \gamma'(t)= J \, B(t)\,\gamma(t),\qquad t \in [0,T]\\
  \gamma(0)=\Id
 \end{cases}
\end{equation}
where $B(t)\= D^2 H\big(t, z(t)\big) $. Now, to each $z$ we will associate an integer
called the  {\em geometrical index\/} and defined
in terms of the Maslov-type index $\iCLM$ introduced by authors in \cite{CLM94}.
\begin{note}
We will denote by $V_\pm(*)$ the positive and negative spectral space  of the
operator $*$.
\end{note}
\begin{defn}\label{def:Maslov-solution-intro}
We define the
{\em geometric index  of a solution $z$\/} of the Hamiltonian
System given in Equation \eqref{eq:Ham-sys-intro} as
 \[
 \iMas(z)\= \iCLM\big(L, \Gr(\gamma); [0,T]\big)
\]
where  $L=\Gr(Q)$.

Analogously, by setting
 $L^\pm=V_\pm(M\,N)\times V_\pm(N \, M^{n-1}) $, we define  the {\em positive\/} and {\em negative geometric indices\/}, as follows
 \[
 \iMas^\pm(z)\= \iMas\big(L^\pm, \Gr(\gamma); t\in[0,T/2]\big).
\]
\end{defn}
\begin{rem}
Being $L^\pm$ Lagrangian subspaces (cf. Lemma \ref{thm:lagrangian}), it readily follows that
Definition \ref{def:Maslov-solution-intro} is well-given.
\end{rem}
\begin{mainthm}({\bf A $\die{n}$-equivariant  Bott-type formula for Hamiltonian Systems)}
\label{thm:Bott-diedrale-Hamiltoniano-intro}
Let $\gamma$ be the (fundamental) solution of the Hamiltonian system given in Equation \eqref{eq:i1}
and let $\widetilde \gamma\= \diag(\gamma, \gamma)$. For
$k=0, \dots \bar n$, we denote  by $P$  and $Q_k$ the matrices
respectively defined by
$P\= \begin{bmatrix} 0 & MN\\ MN & 0 \end{bmatrix}$ and by  $Q_k\=
\begin{bmatrix} 0 &\zeta_n^{-k(n-1)}N\\ \zeta_n^{k(n-1)} N & 0 \end{bmatrix}$.
Then, we have
 \[   \iMas(z)= \iMas^+(z)+ \iMas^-(z).
 \]
 Moreover, if
\begin{itemize}
\item[(i)]   {\bf $n$  odd\/}
\begin{multline}
\iMas^\pm(z)= \iCLM\left(V_\pm(N), \gamma(t) V_\pm(MN); t \in \left[0,
\dfrac{T}{2n}\right]\right)\\
+ \sum_{k=1}^{\frac{n-1}{2}} \iCLM\left(V_\pm(Q_k), \widetilde \gamma(t) V_\pm(P);
t \in \left[0, \dfrac{T}{2n}\right]\right).
\end{multline}
\item[(ii)] {\bf $n$  even\/}
\begin{multline}
\iMas^\pm(z)= \iCLM\left(V_\pm((-1)^{n-1}N), \gamma(t) V_\pm((-1)^nMN); t \in \left[0,
\dfrac{T}{2n}\right]\right)\\
+  \iCLM\left(V_\pm(N), \gamma(t) V_\pm(MN); t \in \left[0,
\dfrac{T}{2n}\right]\right)\\
+\sum_{k=1}^{\frac{n}{2}-1} \iCLM\left(V_\pm(Q_k), \widetilde \gamma(t)
V_\pm(P); t \in \left[0, \dfrac{T}{2n}\right]\right).
\end{multline}
\end{itemize}
\end{mainthm}
We now  assume that  $Q,S,N\in\OO(m) $ be  such that $S^n=Q$,  $N^2=I_m$,   $SN=N\trasp{S}$ and we define the
Hilbert spaces:
\[
E_\R=\Set{x\in W^{1,2}(\R/(T\Z), \R^m)|
x(t)=Qx(t+T)} \textrm{ and } E=E_\R\otimes\C.
\]
Given the Lagrangian function  $L \in \mathscr C^2\big([0,T]\times \R^{n}, \R\big)$
satisfing the following two properties
\begin{equation}\label{eq:lagrangian-properties-intro}
L(t,x,v)=L(t-T/n, Sx,Sv) \quad \textrm{ and } \quad L(t,u,v)=L(T/n-t,Nx,Nv),
\end{equation}
we associate the Lagrangian action functional $\mathscr S_L: E_\R\to \R$,  defined by
\begin{equation}\label{eq:lagr-action2-intro}
 \mathscr S_L(x)\=\int_0^T L\big(t, x(t), \dot x(t)\big)\, dt.
\end{equation}
We observe that the relations given in Equation \eqref{eq:lagrangian-properties-intro} insure
that $ \mathscr S_L$ is $\die{n}$-equivariant and up to standard regularity arguments  $x$
is a critical point of $\mathscr S_L$ if and only if it is a  classical
solution of the following boundary value problem
\begin{equation}
 \begin{cases}
   \dfrac{d}{dt}\dfrac{\partial L}{\partial v}(t,x,\dot{x})-\dfrac{\partial
L}{\partial u}(t,x,\dot{x})=0, \qquad t \in [0,T]\\
x(0)=Qx(T)  \textrm{ and } \dot{x}(0)=Q\dot{x}(T).
 \end{cases}
\end{equation}
Let  $x$ be a  $D_n$-equivariant critical point; thus $x \in E^{\die{n}}$.  By the second variation of
the Lagrangian action, we get that the index form of $x$ is given by
\begin{equation}
\mathcal {I}(u,v)=\int_0^{T}\{\langle(P\dot{u}+Qu), \dot v\rangle+\langle \trasp{Q}\dot u, v
\rangle +\langle R u,v \rangle \}dt, \qquad   u,v\in E
\end{equation}
where ${\displaystyle P(t)\=\dfrac{\partial^2 L}{\partial
v^2}(t,x(t),\dot{x}(t))}$, ${\displaystyle Q(t)\=\dfrac{\partial^2
L}{\partial u\partial v}(t,x(t),\dot{x}(t))}$ and ${\displaystyle
R(t)\=\dfrac{\partial^2 L}{\partial u^2}(t,x(t),\dot{x}(t))}$.
We set
\[
\widehat{\mathcal   A}=-\dfrac{d}{dt}\left(P(t)\dfrac{d}{dt}+Q(t)\right)+Q^T(t)\dfrac{d}{dt}+R(t)
\]
and we observe that, since $x \in E^{\die{n}}$, then  $\widehat{\mathcal A}$ is $D_n$-equivariant.
By the $D_n$-equivariance of $\mathscr S_L$, we immediately get
\begin{multline}\label{eq:SP-intro}
SP(t+\dfrac{T}{n})=P(t)S,\quad  SQ(t+\dfrac{T}{n})=Q(t)S, \quad
SR(t+\dfrac{T}{n})=R(t)S,\\
NP(\dfrac{T}{n}-t)=P(t)N,\quad  NQ(\dfrac{T}{n}-t)=Q(t)N, \quad
NR(\dfrac{T}{n}-t)=R(t)N.
\end{multline}
\begin{mainthm}({\bf A $\die{n}$-equivariant Bott-type formula for Lagrangian systems\/})
\label{thm:Bott-diedrale-Lagrangiano-intro}
Under the previous notation, for any $h=0,\cdots,n-1$, we have
 \[
\iMor(x)= \iMor\left(\mathcal
{I}|_{E^+_h}\right)+ \iMor\left(\mathcal
{I}|_{E^-_h}\right).
\]
Moreover,
\[
\iMor\left(\mathcal{I}|_{E^\pm_h}\right)=  \sum_{k=0}^{\bar n} \iMor\left(\mathcal
{I}|_{F^\pm_{k,h}}\right).
\]
\end{mainthm}

\paragraph{Morse indices of the Figure-eight orbit for the planar 3BP.}
As promised in the beginning the first application of the Bott dihedral invariant formulas is the computation
of the $G$-equivariant Morse index for the figure-eight orbit in the planar three body problem.  For,
we decompose the path space of the collisionless configuration space
of the planar three body problem with equal masses into $\die{6}$-invariant closed stable
subspaces and we compute the Morse index of the restriction of the
second variation of the Lagrangian action onto these subspaces.

We consider three point-masses particles $x_i \in \R^2$ (thus,
$(x_1, x_2 ,x_3) \in (\R^2)^3$) self-interacting with the Newtonian
gravitational potential $U$ and having unit masses. Newton equations are given
by
\[
\dfrac{d^{2}x_{i}}{dt^{2}}=\dfrac{\partial U}{\partial
x_{i}}(x_{1},x_{2},x_{3}),\qquad  i=1,2,3
\]
where
\[
U(x_{1},x_{2},x_{3}) \= \sum_{i<j}\dfrac{1}{|x_{i}-x_{j}|}
\]
is the  potential  function.  The {\em configuration space \/} of the system having centre
of mass in $0$  is given by
\[
\mathcal X\=\{x=(x_{1},x_{2},x_{3})\in(\R^{2})^{3}|x_{1}+x_{2}+x_{3}=0\}.
\]
For each pair of indices $i,j \in \{1,2,3\}$, let $\Delta_{i,j}\=\Set{x \in \mathcal X|x_i=x_j}$ be the
{\em collision set of the i-th and j-th particle\/} and let
\begin{equation}
 \Delta\=\bigcup_{\substack{{i,j=1}\\ i \neq j}}^3 \Delta_{i,j}
\end{equation}
be the {\em collision set in $\mathcal X$\/} (actually an arrangement of hyperplanes). We also define the {\em collisionless
configuration space\/} as   $\widehat {\mathcal X}\= \mathcal X \setminus \Delta$.
For a fixed period $T$, we consider on the Sobolev space
$E\= W^{1,2}(\R/(T\Z),\widehat{\mathcal X})$
the Lagrangian action functional 
\begin{equation}\label{eq:laphi}
\Phi:W^{1,2}(\R/T\Z,\widehat{\mathcal X})\longrightarrow \R \textrm{ defined by }
 \Phi(x)\=\int\limits_0^T L(x(t),\dot{x}(t))dt
\end{equation}
where $ L(x(t),\dot{x}(t))=\dfrac{1}{2}|\dot{x}(t)|^{2}+U(x(t)) $ is the
Lagrangian function.  It's well-known that the figure-eight orbit can be seen as the minimizer
of the action functional on the  $ D_{6}$-iequivariant loop space (cf.
\cite{Che02a, Che02b, CM00} and references therein),   where
\[
\die{6}=\langle g_{1},g_{2}|g^{6}
_{1}=I_{6},g^{2}_{2}=I_{6},g_{1}g_{2}=g_{2}g_{1}^{-1}\rangle
\]
For any $ x=(x_{1},x_{2},x_{3})\in
W^{1,2}(\R/(T\Z),(\R^{2})^{3})$, we assume that
the generators  $g_{1} $,  $ g_{2} $  of $ D_{6}$ act on $
W^{1,2}(\R/(T\Z),(\R^{2})^{3}) $ as follows:
\[
(g_{1} x)(t)\=-RSx(t+\dfrac{T}{6}) \qquad  (g_{2} x)(t)\=RNx(\dfrac{T}{6}-t)
\]
where $ S=\begin{bmatrix}
0&0&I_{2}&\\I_{2}&0&0&\\0&I_{2}&0&\end{bmatrix} $,
$ N\= \begin{bmatrix} I_{2}&0&0&\\0&0&I_{2}&\\0&I_{2}&0&\end{bmatrix}
$ and $R$ is given by
$ R\=\begin{bmatrix} R_{2}&0&0&\\0&R_{2}&0&\\0&0&R_{2}&\end{bmatrix}
$ with
$ R_{2}\= \begin{bmatrix}1&0\\0&-1\end{bmatrix} .$
(For further details, we refer the interested reader to \cite[Section 5]{HS09}).

For any $h=0, \dots , 5$  the following orthogonal direct sum decomposition  of the $G_h$-equivariant
space $E$ holds
\[
 E_h=E_{h}^{+}\bigoplus E_{h}^{-}
 \]
where
\[
E_{h}^{+}\=\bigoplus\limits_{k=1}^{2}F_{k,h}^{+}\oplus
F_{0,h}^{+}\oplus F_{3,h}^{+} \quad \textrm{ and } \quad
E_{h}^{-}=\bigoplus\limits_{k=1}^{2}F_{k,h}^{-}\oplus F_{0,h}^{-}\oplus
F_{3,h}^{-}.
\]
Being the figure-eight orbit $x$ a $\die{6}$-equivariant critical point and by invoking
the Palais principle of symmetric criticality, it follows that $x$ is a critical point for the restriction of
$\Phi$ to the space  fixed by the action of $G_h$ (that in shorthand notation will be denoted by $\Phi^h$),
where $G_h \subset \die{n}$ is the subgroup generated by
$\mathcal S \mathcal R^h$, where $\mathcal R $ and $\mathcal S$ are the generators of $\die{n}$.
\begin{note}
In  what follows, we denote by $\iMorh{h}(x)$ the Morse index of $x$ as critical point of the restricted  map $\Phi^h$; thus
the Morse index of the second variation of $\Phi^h$ on the space $E^{G_h}$.
\end{note}

 \begin{mainthm}\label{thm:main-3bp}
Under the notation above, the Morse index $\iMorh{h}(x)$ of the Figure-eight orbit in the $G_h$-equivariant space $E_h$
 is given by
\[
\iMorh{h}(x)= \sum \limits_{k=0}^{3} \Big[\iMor(F_{k,h}^{+})+\iMor(F_{k,h}^{-})\Big],
\]
where we denoted by the symbol $\iMor(F_{k,h}^{\pm})$ the Morse index of the restriction of the
second variation of $\Phi$ on $ F_{k,h}^{\pm}$ defined in Equation \eqref{eq:i0}. Furthermore, we have
 \begin{multline}
  \iMor(F^\pm_{0,h})= \iMor(F^\pm_{3,h}) =\iMor(F^\pm_{2,h})=0, \quad   \iMor(F^\pm_{1,h})=1 \qquad \textrm{ for }
  h=0,\cdots,5 \textrm{ and }\\
  \iMor(E_2)=\iMor(E_4)=\iMor(E_4)=0, \quad \iMor(E_1)=\iMor(E_5)=1.
    \end{multline}
    \end{mainthm}
    \begin{rem}
Being a minimizer on the $\die{6}$-equivariant loop space of
the collision manifold, this in particular implies that $\iMor(F_{0,0}^{+})=0$. (In fact,  the restriction of $\mathcal R$ on $E_0$ is the identity
as well as the restriction of $\mathcal S$ on $F_{0,0}^+$). We point out that
the claims  $\iMor(x)=2$, $\iMorh{2}(x)=\iMorh{3}(x)=0$, follows directly from
 \cite[Remark 5.12]{HS09}; it is worth noticing that these results
 were numerically obtained by using Matlab.
\end{rem}
\paragraph{A hyperbolicity criterion for reversible Lagrangian systems.}
The last application of the theory developed is related to study the strongly instability and hyperbolicity
of dihedral equivariant critical points. For, let $L \in \mathscr C^2([0,T] \times \R^{2m}, \R)$  be a Lagrangian function
 satisfying the {\em Legendre convexity condition\/} and, as before, we
consider the Lagrangian action functional
\begin{equation}
\mathscr S_L:E_\R \to \R \textrm{ defined by } \mathscr S_L(x)=\int_0^T L\big(t, x(t), \dot x(t)\big)\, dt,
\end{equation}
where $E_\R\=W^{1,2}([0,T]; \R^m)$. If $x$ is a critical point, by the second variation of $\mathscr S_L$
we get the associated index form given by
\begin{equation}
\mathcal {I}_{\R}(\xi,\eta)=\int_0^{T}\Big[\langle(P\dot{\xi}+Q\xi),\dot{\eta}\rangle+\langle \trasp{Q}\dot{\xi},
\eta\rangle+\langle R\xi, \eta\rangle\Big]dt, \qquad  \xi,\eta\in E_\R.
\end{equation}
For any $\omega \in \U$, we let
\[
E^{1}(\omega)\=\Set{u \in E | u(0)=\omega\, u(T)}
\]
where we denoted by $E$  the complexification of $E_\R$; i.e. $E\=E_\R \otimes \C$. The corresponding $\omega$-index
form is given by
\begin{equation}\label{eq:index-form-complex-intro}
\mathcal {I}_\omega(\xi,\eta)=\int_0^{T}\Big[\langle(P\dot{\xi}+Q\xi),\dot{\eta}\rangle+\langle \trasp{Q}\dot{\xi},
\eta\rangle+\langle R\xi,  \eta\rangle \Big]dt, \qquad   \xi,\eta\in E
\end{equation}
where we extended  the Euclidean product  in $\R^m$,  $\langle\cdot, \cdot \rangle$ to the standard Hermitian product in $\C^m$.
 In what follows we will  denote by $\iMor(\omega,x)$ the Morse
index of $\mathcal I_\omega$ on $E^1(\omega)$.
Let us consider   the linear Sturm differential operator
\begin{equation}\label{eq:MS-system-intro}
\mathcal A \xi\=-\dfrac{d}{d t} \big(P(t) \dot \xi + Q(t) \xi\big) + \trasp{Q}(t)\dot \xi + R(t) \xi, \qquad t \in [0,T].
\end{equation}
By using the Legendre transform we reduce the second order system $\mathcal A \xi =0$ to the linear Hamiltonian system
\begin{equation}\label{eq:hamsys-2}
\dot z = J B(t) z, \qquad t \in [0,T]
\end{equation}
where
\[
B(t)\= \begin{bmatrix}
P^{-1}(t) & - P^{-1}(t) Q(t)\\
-\trasp{Q}(t)P^{-1}(t)& \trasp{Q}(t) P^{-1}(t)Q(t)-R(t)
\end{bmatrix}.
\]
Let $\gamma$ be the fundamental solution of the Hamiltonian system given in Equation \eqref{eq:hamsys-2} and
let $M_{T}\=\gamma(T)$ be the induced monodromy matrix. Let $\varepsilon >0$ and let us denote
 by $\mathscr U_\varepsilon(M_T)$ the $\varepsilon$ neighborhood of $M_T$ in $\Sp(2m)$
(with respect to the operator topology induced by $\Lin(\R^m)$.
\begin{defn}\label{def:strongly-stable-intro}
A $T$-periodic orbit of the Hamiltonian system given in Equation \eqref{eq:hamsys-2}
is termed {\em strongly stable\/} if there exists $\varepsilon >0$ such that
\[
\textrm{ for every } M \in \mathscr U_\varepsilon(M_T) \Rightarrow M \textrm{ is linearly stable. }
\]
\end{defn}
The next result put on evidence the prominent role played  by the Neumann boundary condition with respect to the other selfadjoint boundary conditions on the
Sturm operator $\mathcal A$ for controlling the strong stability and  the hyperbolicity.
\begin{mainthm}\label{thm:1-ultimasezione-intro}
Let $S, N \in \OO(m)$ be such that $S^{n}= \Id_{m}$ and $N^{2}=(NS)^{2}=\Id_{m}$ acting on $E_\R$ dihedrally through  the
action given by
\begin{multline}\label{eq:action-2-intro}
\mathcal S: E \ni  z\longmapsto  (\mathcal S\, z)(\cdot) \=S\, z\left(\cdot+ T/n\right)\in E_\R
\textrm { and }\\
\mathcal N: E \ni  z\longmapsto  (\mathcal N\, z)(\cdot) \=N\, z\left(T/n-t\right)\in E_\R.
\end{multline}
Under the above notation,  if f the restriction of $\mathcal A$ given in Equation \eqref{eq:MS-system-intro} on
\begin{equation}
E^2(T/(2n))\=\Set{u \in W^{2,2}([0,T/(2n)], \C^{m})| \dot{u}(0)=\dot{u}(T/(2n))=0}
\end{equation}
is positive semi-definite, then $x$ can not be strongly stable.
\end{mainthm}
\begin{maincor}\label{thm:2-ultimasezione-intro}
Under the same assumptions of Theorem \ref{thm:1-ultimasezione-intro} and if
$\mathcal A$ is positive definite on $E^2(T/(2n))$, then $x$  is hyperbolic.
\end{maincor}
We close this section by observing that as a direct consequence of Corollary \ref{thm:2-ultimasezione-intro}, we
provide  a completely new proof of \cite[Pag.627, Theorem 2.9]{Off92}.
More precisely, let us consider the second order linear differential operator
\[
\mathcal L= -\dfrac{d^2}{dt^2} + Q(t)
\]
where $Q$ is a $T$-periodic path of symmetric matrices. We recall that the  operator $\mathcal L$ is termed {\em reversible} if
\[
Q(-t)=Q(t) \qquad t \in [0,T].
\]
(Cf., for further details,  \cite{Off92} and references therein).
\begin{maincor}[{\bf Offin 1992\/}] \label{thm:Offin}
We assume that the operator $\mathcal L$ is reversible, non-degenerate (meaning that
 there are no symmetric $T$-periodic solutions of the
equation $\mathcal L u=0$) and have a vanishing Morse index.
Then the associated monodromy matrix $M_T$ is hyperbolic.
\end{maincor}



\section{Equivariant preliminaries and symmetry constraints}\label{sec:reptheory}

This section is devoted to introduce the basic definitions and notation and to
describe the  results that we will need in the sequel. Our
basic references for the material contained in  this section are
\cite{Ser77,Pal79, Kow, Mac89} and references therein.

\subsection{Equivariant preliminaries}\label{subsec:preliminaries}

Let $G$ be a finite group acting on a (representation) space $X$. The space $X$
is called {\em $G$-equivariant space\/}. We recall that the   {\em isotropy group\/}
or the {\em fixer\/} of $x$ in $G$ is defined as $ G_x\=\Set{g \in G | gx=x}$.
If $H \subset G$ is a subgroup,
the space $X^H \subset X$ consists of all points $x \in X$ which are fixed by $H$.
Given two $G$-equivariant spaces $X$ and $Y$, an {\em equivariant map
$f: X \to Y$\/} is a map having the property that $f(g \cdot x)= g\cdot f(x)$  for all  $g \in G$ and  for all $x \in X$.
We observe that an equivariant map $f:X \to Y$ induces by restriction to the $X^H$ fixed by the
subgroup $H \subset G$ a map $f^H: X^H \to Y^H$. Let $E$ be a separable complex Hilbert space,
with identity $\Id_E$ and let $G$ be a  topological compact group.
By a {\em unitary representation\/} $U$ of such a group $G$ we shall
mean a mapping $G \ni g \longmapsto U_g \in \U(E) $
where $\U(E)$ denotes the group of {\em unitary operators\/} of  $E$, which has the following two properties:
\begin{itemize}
 \item[(i)] $U_{g_1 g_2}=U_{g_1}U_{g_2}$ for every $g_1, g_2 \in G$;
 \item[(ii)] $G \ni g \mapsto U_g(x) \in H$ is a continuous function from $G$ to $H$ for every $u \in H$.
\end{itemize}
\begin{rem}
It is worth noticing that any arbitrary group $G$ equipped with the discrete
topology can be turned into a topological group. Moreover  in this case the continuity is automatic. Furthermore
if $G$ is finite, it is clearly compact.
\end{rem}
For the sake of the reader we briefly recall  in a
very special situation which is sufficient for the theory developed in this paper,  the classical {\em Palais principle of symmetric criticality\/}
proved by author in \cite{Pal79}.

Let $G$ be a finite group acting on $E$ through unitary transformations, $f: E \to \R$ be a
$G$-invariant functional of $\mathscr C^2$-class and let us denote by $f^G$ the map induced by restriction of $f$
on the (closed) subspace $E^G$ fixed by the action of $G$.
\begin{lem}\label{thm:palais}({\bf Palais Principle of Symmetric
Criticality\/})
If $f$ is $G$-invariant then a critical point of $f^G$ in $E^G$ is a critical point of
$f$ in $E$.
\end{lem}
\begin{proof}
Since $f$ is $G$-invariant then, for every $g \in G$ we have  $f (U_g\, x)=f(x)$ for
every $x \in E$, where $U_g$ denotes the unitary representation of $g$. Denoting by $\nabla$ the gradient and by
$d$ the differential,  we get
\begin{multline}
df(x)[v]=\dfrac{d}{d\varepsilon}\Big\vert_{\varepsilon=0}f(x+
\varepsilon\, v)=
 \langle \nabla f(x), v\rangle  \\
 d\,f(U_g\,x)[v]=\dfrac{d}{d\varepsilon}\Big\vert_{\varepsilon=0}f\big(U_g(x+ \varepsilon\, v)\big)=
 \langle \nabla f(U_g x), U_g v \rangle=  \langle U_g^*\nabla f(U_g x),
v \rangle,
 \quad \forall\, v \in E
\end{multline}
where we denoted by $U_g^*$  the conjugate transpose (or Hermitian transpose) of
$U_g$. In particular
\begin{equation}\label{eq:chain-rule}
 U_g^*\, \nabla f(U_g\, x)= \nabla f(x)\qquad  x \in
E.
\end{equation}
Thus
\begin{equation}\label{eq:quella-che-si-usa}
 U_g\,\nabla f(x)=\nabla f(U_g\,x) \qquad g \in G, \ \ x \in
E
\end{equation}
and by this it readily follows that if $x \in E^G$ then $\nabla f(x)\in
E^G$ or which is the same $\nabla f(x)=\nabla f^G(x)$.
\\
($\Rightarrow$). In order to prove the  if part, we assume that $x \in E^G$ is
a critical point of $f^G$.
Thus, we have $U_g\, x =x$
and $\nabla f^G(x)=0$. Thus, by Equation \eqref{eq:quella-che-si-usa},
we have
\begin{equation}
 0= \nabla f^G(x)=\nabla f^G(U_g\,x)=\nabla f(U_g\,x)
 = U_g\,\nabla f(x)
\end{equation}
which implies that $x$ is  a critical point of $E$.\\
($\Leftarrow$). The proof of the only if part, readily follows by the previous
arguments.
(For further details, we refer the interested reader to \cite{Pal79}). This
conclude the proof.
\end{proof}
The next result put on evidence a commutativity property enjoyed by the Hessian of a $G$-invariant
functional $f$ at a critical point for $f^G$.
\begin{lem}
 Let $x \in E^G$ be a critical point of $f^G$. If $\Hess  f(x)$ denotes the Hessian matrix
 of  $f$ at $x$, the following commutativity property holds
 \[
  U_g\, \Hess f(x)= \Hess f(x)\, U_g, \qquad \forall \, U_g \in \U(E)
 \]
\end{lem}
\begin{proof}
 By the very same calculations as above, we have for all $v, w \in E$
 \begin{multline}\label{eq:prova-mult}
 d^2\,f(x)[v, w]=\dfrac{d}{d\varepsilon}\Big\vert_{\varepsilon=0}\langle
\nabla
 f(x+ \varepsilon\, w), v \rangle =\langle \Hess f(x)\, w,
v\rangle\\
 d^2\,f(U_g\,x)[U_g\,v ,U_g\,
w]=\dfrac{d}{d\varepsilon}\Big\vert_{\varepsilon=0}\langle
 \nabla f\big(U_g(x+ \varepsilon\, w)\big),U_ g\, v \rangle \\=
 \langle \Hess f(U_g x)\, U_g w, U_g\, v \rangle=
 \langle \Hess f(x)\, U_g w, U_g\, v \rangle=
 \langle U_g^*\Hess f(x) U_g\,w,v \rangle.
\end{multline}
Since, by  the chain rule $\forall\, v , w \in E,$ $d^2\,f(x)[v, w]=d^2\,f(U_g\,x)[U_g\,v, U_g\, w]$
by the previous computation, we immediately get $\Hess f(x)=   U_g^*\Hess
f(x) U_g$ and this conclude the proof.
\end{proof}
We are now ready to prove  an  abstract
{\em $\die{n}$-equivariant spectral flow formula\/} for $\die{n}$-equivariant paths of
linear and bounded selfadjoint Fredholm operators  on the (complex) Hilbert space $E$.

With a slight abuse of notation, we denote with the same symbol the dihedral
group as well as its unitary representation in $E$. Thus, the group $\die{n}$
(actually  a unitary  representation of the dihedral group
$\die{n}$) is presented as follows
\begin{equation}\label{eq:image-diedrale-unitaria}
 D_n \= \langle \mathcal R,  \mathcal  S \in \U(E)| \mathcal R^n=
  \mathcal S^2=( \mathcal  S  \mathcal  R)^2= \Id_E\rangle\subset \U(E).
\end{equation}
Let  $ \mathcal A: [0,1] \to \fredsa(E)$ be a continuous
path of closed selfadjoint Fredholm operators commuting with $\mathcal R$ and $\mathcal
S$; in symbols, we have
\[
 \mathcal A(\lambda)\,{\mathcal R}= {\mathcal R}\, \mathcal A(\lambda) \textrm{ and }
  \mathcal  A(\lambda)\, \mathcal S=  \mathcal S\, \mathcal A(\lambda) \textrm{ for all }
\lambda \in [0,1].
\]
As already showed in Section \ref{sec:intro}, we can decompose the  Hilbert space $E$ into mutually orthogonal
$\die{n}$-stable modules, given by
\begin{equation}\label{eq:decomposition-2}
 E=F_0\oplus \dots \oplus F_{\bar n}
\end{equation}
where $F_k$  were defined in Equation \eqref{eq:Fk-intro}. Now, for each $k =0, \dots, \bar n$, let  $\mathcal  A_k$ be the continuous path
defined by
$  \mathcal A_k\= \mathcal  A\big\vert_{ F_k}:[0,1] \to  \fredsa(
F_k)$. Thus, we have
$ \mathcal A(\lambda)=\mathcal  A_0(\lambda) \oplus \dots \oplus \mathcal  A_{\bar n}(\lambda),$ for $ \lambda\in [0,1]$.

For $k=0, \dots, \bar n$ and $h=0, \dots, n-1$, we define the
closed subspaces
\[
F_{k,h}^\pm=\Set{u \in F_k| \mathcal S\mathcal  R^h\, u = \pm \, u}.\]
Since $ R|_{E_0}=I$,  we get that  $ F_{0,h}^\pm=F_{0,0}^\pm=\Set{u \in F_0|
 \mathcal S\, u = \pm \, u} . $
In the case  $n$ is even, $ \mathcal R|_{E_{n/2}}=-\Id$, so we have
$F_{n/2,h}^\pm=\Set{u \in F_{n/2}| (-1)^h\mathcal  S\, u = \pm \, u} .$
It is worth noticing  that, for $k=1, \dots, \myfloor{(n-1)/2}$, we get
$F_{k,h}^\pm= \Set{\begin{bmatrix} I\\\mathcal  S
 \mathcal R^h\end{bmatrix}\, z| z \in E_k}.$
We also observe that for any $h=0, \dots, n-1$ we have
\[
\big( \mathcal S\mathcal R^{h}\big) K(\lambda)= K(\lambda),
\textrm{ where }
\mathcal K(\lambda) \= \ker \mathcal A(\lambda),\qquad
 \lambda \in [0,1].
\]
By the direct sum property of the spectral flow, we have
 \[
\begin{split}
 \spfl(\mathcal A|_{F_k};[0,1])&=
\spfl( \mathcal A|_{E_k};[0,1])+\spfl( \mathcal A|_{E_{-k}};[0,1])\\&=
\spfl(\mathcal  A|_{F^+_{k,h}};[0,1])+\spfl( \mathcal A|_{F^-_{k,h}};[0,1]).
 \end{split}
 \]
\begin{prop}\label{thm:new-spfl-formulae}
For every $k=1, \dots, \myfloor{(n-1)/2}$ and $h=0, \dots n-1$, we have
 \[
 \begin{split}
\spfl( \mathcal A|_{F_k};[0,1])&=
2\spfl(\mathcal A|_{E_k};[0,1])=2\spfl(\mathcal A|_{E_{-k}};[0,1])\\&=
2\spfl(\mathcal A|_{F^+_{k,h}};[0,1])=2\spfl(\mathcal A|_{F^-_{k,h}};[0,1]).
\end{split}
\]
In particular
\[
\spfl( \mathcal A|_{F_k};[0,1])\equiv 0 \mmod \qquad  \textrm{ for every } k=1, \dots, \myfloor{(n-1)/2}.
\]	
\end{prop}
\begin{proof}
We start by observing that for each $\lambda \in [0,1]$, it holds
\[
 \mathcal A_\lambda|_{F_k}= \begin{pmatrix}  \mathcal A_\lambda|_{E_k} & 0 \\ 0 &
 \mathcal A_\lambda|_{E_{-k}}
\end{pmatrix}.
\]
Moreover it is immediate to check that
\[
 z \in E_k, \ \  \mathcal  A_\lambda \, z =0 \quad \iff \quad \bar z \in E_{-k},\
\ A_\lambda \bar z =0
\]
or otherwise stated $\ker \mathcal  A_\lambda|_{E_k} = \overline{\ker
 \mathcal A_\lambda|_{E_{-k}}}$. In particular, they have the
same dimension. Moreover, if $\lambda_0 \in [0,1]$ is a crossing instant, then,
for each $k =0, \dots,\myfloor{(n-1)/2}$, it is
immediate to verify that  the crossing operators across $\lambda_0$ both coincide; thus, we have
\[
 \Gamma \left(  \mathcal A_\lambda|_{E_k}, \lambda_0 \right)=
 \Gamma\left( \mathcal A_\lambda|_{E_{-k}}, \lambda_0\right).
\]
By taking into account the perturbation result (cfr. \cite[Theorem 2.6]{Wat15}),
we  don't lead in generalities in assuming that the path $ \mathcal A$
is regular. By taking the sum over all crossing points of the signature of the
crossing forms (since we are assuming that
the path $ \mathcal A$ is regular), we immediately get
\[
 \spfl( \mathcal A|_{E_k}; [0,1])= \spfl( \mathcal A|_{E_{-k}}; [0,1]).
\]
In order to conclude the proof, we start to observe that if $ z\in \ker
 \mathcal A_\lambda $ then
$ \mathcal S\,  \mathcal R^h z \in \ker  \mathcal A_\lambda $. So
\[
\widetilde{ \mathcal A}_\lambda\= \begin{bmatrix}
   \mathcal A_\lambda &0\\ 0 & \mathcal  A_\lambda
 \end{bmatrix}\, \begin{bmatrix} z\\  \mathcal S\,  \mathcal R^h z \end{bmatrix}=0
\iff
\begin{bmatrix}
 \mathcal   A_\lambda &0\\ 0 &  \mathcal A_\lambda
 \end{bmatrix}\, \begin{bmatrix} z\\ - \mathcal S\,  \mathcal R^h z
\end{bmatrix}=0
 \]
meaning that $\ker\left( \mathcal A_\lambda|_{F_{k,h}^+}\right)= \ker\left(
 \mathcal A_\lambda|_{F_{k,h}^-}\right)$ for all
$\lambda \in [0,1]$. We denote by $\widetilde \Gamma\left(
 \mathcal A_\lambda|_{F_{k,h}^\pm}, \lambda_0 \right)$ the
crossing operator onto $\ker\left( \mathcal A_{\lambda_0}|_{F_{k,h}^\pm}\right)$, namely
\[
 \widetilde \Gamma\left( \mathcal A_\lambda|_{F_{k,h}^\pm}, \lambda_0 \right) : \ker\left(
 \mathcal A_\lambda|_{F_{k,h}^\pm}\right)\longrightarrow \R
\]
given by
\[
\widetilde \Gamma\left( \mathcal A_\lambda|_{F_{k,h}^\pm}, \lambda_0 \right)
\begin{bmatrix}
        z\\
        \pm  \mathcal S\,  \mathcal R^h z
       \end{bmatrix}=\big\langle
\dot{ \mathcal A}_{\lambda_0}\, z, z \big \rangle_E + \big \langle ( \mathcal S\,
 \mathcal R^h)^*
\dot{ \mathcal A}_{\lambda_0}\, ( \mathcal S\,  \mathcal R^h) z, z \big \rangle_E.
\]
Summing over all (regular) crossing instants, we get
\[
\spfl\left( \mathcal A|_{F_{k,h}^+}; [0,1]\right)=\spfl\left(
 \mathcal A|_{F_{k,h}^-}; [0,1]\right).
\]
The second claim is a direct consequence of the previous computations. This conclude the proof.
\end{proof}
\paragraph{Proof of Proposition \ref{thm:parity-equi}.}
By invoking Proposition \ref{thm:new-spfl-formulae}, we already know that for every $k=1,  \dots, \myfloor{(n-1)/2}$,
$\spfl( \mathcal A|_{F_k};[0,1])$ is even. Now the result easily follows by taking into account the decomposition given in
Equation \eqref{eq:decomposition-2}. This conclude the proof. \qed


\section{A dihedral-equivariant decomposition of the path
space}\label{subsec:dihedral-loop}

The aim of this paragraph is to construct an equivariant decomposition of the
loop space of the configuration manifold and to establish a dihedral equivariant
Bott-type iteration  formula
in terms of the Maslov index of suitable induced continuous paths of Lagrangian
subspaces.

Let  $\TT \subset \R^2$ be a circle in $\R^2$ of length $T=|\TT|$ which
can be identified with $\R/(T\Z)$, where $T\Z$ denotes the lattice generated by $T \in \R$. Given
$Q \in  \U(2m)$ let $E$ be denote the Sobolev completion of the
set of smooth maps $z:\TT \to \C^{2m}$ such that $z(t)=Q z(t+T)$; thus
\begin{equation}\label{eq:Q-loopspace}
E\=\Set{z \in W^{1,2}(\TT;\C^{2m})| z(t)= Q \, z(t+T)}.
\end{equation}
Let  $M \in  \U(2m)$  and $N \in \U(2m)$ be such that the following commutativity
properties holds
 \begin{equation} \label{eq:relazioni}
M^n=Q,\ \   N^2 = \Id,  \ \ N \, M^*= M\, N.
         \end{equation}
For any $k=0, \dots, n-1$, we denote by $E_k$   the closed  subspace of $E$  given by
\[
E_k= \Set{z \in E|M\, z\left(t+\dfrac{T}{n}\right)
=\zeta_n^k\, z(t)}.
\]
and for $k=1,\cdots,\myfloor{(n-1)/2}$, we define   $\widehat M_k, \widehat N_k$ the
following block diagonal matrices
\begin{equation}\label{eq:matrici-cappuccio}
 \widehat M_k= \begin{bmatrix}
              \zeta_n^{-k} M & 0\\
              0 & \zeta_n^k M
             \end{bmatrix}\quad  \textrm{ and } \quad
 \widehat N_k= \begin{bmatrix}
              0 & \zeta_n^{-k} \, N\\
              \zeta_n^{k} \, N & 0
             \end{bmatrix}.
\end{equation}
As above, we set $F_0=E_0$,  $F_{n/2}=E_{n/2}$ if $n$ is even, $F_k\= E_k \oplus E_{-k}$
for $k=1,\cdots,\myfloor{(n-1)/2}$ and $F_{n/2}=0$ if $n$ is odd. Thus, we have
\begin{equation}\label{eq:F_k}
\begin{split}
F_k=\Set{u \in W^{1,2}\left(\dfrac{\TT}{n}; \C^{2m} \oplus \C^{2m}\right)|\ \
u\=\begin{bmatrix}
                                                            z\\ w
                                                           \end{bmatrix}, \textrm{ where }
                                                           z \in E_k,\ \  w \in
E_{-k}
                                                          \textrm { and }  u(0)=
      \widehat M_k\, u\left(\dfrac{T}{n}\right)}
      \end{split}.
\end{equation}
We now define the two unitary operators on $E$ as follows
\begin{multline}\label{eq:action-1}
 \mathcal M: E \ni  z\longmapsto  ( \mathcal M\, z)(\cdot) \=M\, z\left(\cdot+ T/n\right)\in E
\textrm { and }\\
\mathcal N: E \ni  z\longmapsto  (\mathcal N\, z)(\cdot) \=N\, z\left(T/n-t\right)\in E.
\end{multline}
By using the properties given in Equation \eqref{eq:relazioni}, it is immediate
to check that $E$ equipped by the action defined in Formula \eqref{eq:action-1}
turns out a $\die{n}$-equivariant space. Moreover, as direct consequence of the relations on the
generators of the dihedral group $\die{n}$, we also get that, for every $h =0, \dots, n-1$,
$(\mathcal N\,\mathcal M^h)^2=\Id$.
In this way we decompose the Hilbert space $E$ defined in Equation \eqref{eq:Q-loopspace}
into a direct sum of $\die{n}$-closed stable subspaces; namely
$E= F_0 \oplus \cdots \oplus F_{\bar n}$ where each $F_k$ is defined in Formula
\eqref{eq:F_k}. By restriction, we get the unitary operators
\begin{multline}
\mathcal M_k\= \mathcal M|_{E_k} \in \U(E_k), \  \mathcal N_k\=
\mathcal N|_{E_k} \in \U(E_k), \
\widehat{\mathcal M}_k\= \begin{bmatrix} \mathcal M_k & 0 \\ 0 &
\mathcal M_k \end{bmatrix}\in \U(F_k)\\ \textrm{ and finally }
\widehat{\mathcal N}_k\= \begin{bmatrix} 0 & {\mathcal N_k}\\ \mathcal N_k & 0
\end{bmatrix} \in \U(F_k).
\end{multline}
\begin{lem}\label{thm:spectrumSRh}
For each $h=0, \dots, n-1$, the spectrum of $\mathcal N\mathcal M^h$ is $\{- 1,1\}$.
\end{lem}
\begin{proof}
This fact readily follows by observing that  for every $h =0, \dots, n-1$,
$(\mathcal N\,\mathcal M^h)^2=\Id$.
\end{proof}
As direct consequence  of Lemma \ref{thm:spectrumSRh},
the spectrum of
\[
\widehat{\mathcal N}\,\widehat{\mathcal M}^h\=
\begin{bmatrix}
0 & {\mathcal N}\, \mathcal M^h\\ \mathcal N\, \mathcal M^h & 0
\end{bmatrix}
\]
as well as $\widehat{\mathcal N}_k\widehat{\mathcal M}_k^h$ is $\{-1,1\}$
for every $h=0, \dots, n-1$ and $k=0, \dots , \bar n$.
Furthermore,
for each $u \in F_k$ and   $k=1,\dots, \myfloor{(n-1)/2}$,  we have
\begin{equation}\label{eq:lunga}
\begin{split}
\big(\widehat{\mathcal N}_k\,\widehat{\mathcal M}_k^h\, u\big)(t)&=
\begin{bmatrix} 0 & {\mathcal N_k}\, \mathcal M_k^h\\ \mathcal N_k\,
\mathcal M_k^h & 0 \end{bmatrix}\begin{bmatrix}z\\ w
\end{bmatrix}(t)= \begin{bmatrix} {\mathcal N_k}\, \mathcal M_k^h \, w\\
{\mathcal N_k}\, \mathcal M_k^h \, z\end{bmatrix}(t)= \begin{bmatrix}
\zeta_n^{-kh} {\mathcal N_k} w\\
 \zeta_n^{kh} {\mathcal N_k}\, z \end{bmatrix}(t)\\&=\begin{bmatrix}
\zeta_n^{-kh} N\,w(T/n-t)\\
 \zeta_n^{kh} N\,z(T/n-t) \end{bmatrix}= \begin{bmatrix} 0 & \zeta_n^{-kh} N\\
\zeta_n^{kh} N & 0
\end{bmatrix}\begin{bmatrix} z\left(T/n-t\right)\\
w\left(T/n-t\right)\end{bmatrix}\\&=\widehat N_{kh}\,\begin{bmatrix}
z\left(T/n-t\right)\\
w\left(T/n-t\right)\end{bmatrix},\qquad t \in [0,T].
\end{split}
\end{equation}
We denote by  $V_+(\widehat{\mathcal N}_k \widehat{\mathcal M}_k^h),
V_-(\widehat{\mathcal N}_k\widehat{\mathcal M}_k^h)$  be the positive and negative spectral spaces of
$\widehat{\mathcal N}_k \widehat{\mathcal M}_k^h$ and we set  $F_{k,h}^\pm\=V_\pm(\widehat{\mathcal N}_k
\widehat{\mathcal M}_k^h)$.
\begin{lem}\label{thm:eigenspaces-Lagrangians}
For $h=0, \dots, n-1$ and  $k=1,\cdots, \myfloor{(n-1)/2}$, we let   $u \in
V_\pm(\widehat{\mathcal N}_k\,\widehat{\mathcal M}_k^h)$.  Thus, we have
\begin{itemize}
\item \begin{multline}
u(0) \in  V_\pm(\widehat N_k \widehat M_k^h) =
\Set{\begin{bmatrix}x\\ \pm\zeta_n^{k(h+1)}\, M\, N\, x\end{bmatrix}| x \in
\C^{2m}
}\\
u(T/(2n)) \in V_\pm (\widehat N_k^h) =
\Set{\begin{bmatrix}x \\\pm\zeta_n^{k(h+1)}\, N\, x\end{bmatrix}| x \in \C^{2m}
}.
\end{multline}
\end{itemize}
Furthermore if
\begin{itemize}
\item $u\in F_{0,h}^\pm$, then   we have
\[
u(0) \in  V_\pm(MN) \textrm{ and }
u(T/(2n)) \in V_\pm (N).
\]
\item if  $u\in F_{n/2,h}^\pm$, then  we have 
\[
u(0) \in  V_\pm ((-1)^{(h+1)}MN) \textrm{ and }
u(T/(2n)) \in V_\pm ((-1)^hN)
\]
\end{itemize}
\end{lem}

\begin{proof}
By Lemma \ref{thm:spectrumSRh} we know that, for each $h=0, \dots ,n-1$,
the spectrum of $\widehat{\mathcal N}_k\,
\widehat{\mathcal M}_k^h$ is $\{+1,-1\}$. Let $u \in  V_+(\widehat{\mathcal
N}_k\,
\widehat{\mathcal M}_k^h)$; namely $\widehat{\mathcal N}_k\,
\widehat{\mathcal M}_k^h u = u$. By Formula \eqref{eq:lunga} it follows
that pointwise it holds that
\begin{equation}
 \begin{bmatrix}z(t)\\w(t)\end{bmatrix}= \begin{bmatrix} 0 & \zeta_n^{-kh} N\\
\zeta_n^{kh} N & 0
\end{bmatrix}\begin{bmatrix} z\left(T/n-t\right)\\
w\left(T/n-t\right)\end{bmatrix}.
\end{equation}
In particular
\begin{equation}
\begin{split}
 \begin{bmatrix}z(0)\\w(0)\end{bmatrix}&= \begin{bmatrix} 0 & \zeta_n^{-kh} N\\
\zeta_n^{kh} N & 0
\end{bmatrix}\begin{bmatrix} z\left(T/n\right)\\
w\left(T/n\right)\end{bmatrix}=\begin{bmatrix} 0 & \zeta_n^{-kh} N\\
\zeta_n^{kh} N & 0
\end{bmatrix}\begin{bmatrix} \zeta_n^{k}\, M^* z\left(0\right)\\
\zeta_n^{-k} M^*w\left(0\right)\end{bmatrix}\\
&=\begin{bmatrix} 0 & \zeta_n^{-kh} N\\  \zeta_n^{kh} N & 0
\end{bmatrix}\begin{bmatrix} \zeta_n^{k} M^* &0\\ 0& \zeta_n^{-k} M^*
\end{bmatrix}\begin{bmatrix}  z\left(0\right)\\
w\left(0\right)\end{bmatrix}\\
&= \begin{bmatrix} 0 & \zeta_n^{-k(h+1)} N\, M^*\\  \zeta_n^{k(h+1)} N\, M^* & 0
\end{bmatrix}\begin{bmatrix}  z\left(0\right)\\
w\left(0\right)\end{bmatrix}= \begin{bmatrix} 0 & \zeta_n^{-k(h+1)} M\, N \\
\zeta_n^{k(h+1)} M\, N & 0
\end{bmatrix}\begin{bmatrix}  z\left(0\right)\\
w\left(0\right)\end{bmatrix}.
\end{split}
\end{equation}
Thus $\begin{bmatrix}z(0)\\w(0)\end{bmatrix}\in \Fix
\begin{bmatrix} 0 & \zeta_n^{-k(h+1)} M\, N \\  \zeta_n^{k(h+1)} M\, N & 0
\end{bmatrix}$. By a direct computation, we get
\begin{equation}
\Fix\begin{bmatrix} 0 & \zeta_n^{-k(h+1)} M\, N \\  \zeta_n^{k(h+1)} M\, N & 0
\end{bmatrix}= \left\{\begin{pmatrix}x\\ \zeta_n^{k(h+1)}\, M\, N
x\end{pmatrix}, x \in \C^{2m}\right\}.
\end{equation}
For $t=T/(2n)$, we have
$\begin{bmatrix}z(T/(2n))\\w(T/(2n))\end{bmatrix}= \begin{bmatrix} 0 &
\zeta_n^{-kh} N\\  \zeta_n^{kh} N & 0
\end{bmatrix}\begin{bmatrix} z\left(T/(2n)\right)\\
w\left(T/(2n)\right)\end{bmatrix}.$
So \begin{equation}
\begin{bmatrix}z(T/(2n))\\w(T/(2n))\end{bmatrix} \in \Fix \begin{bmatrix} 0 &
\zeta_n^{-kh} N\\  \zeta_n^{kh} N & 0
\end{bmatrix}=
\left\{\begin{pmatrix}x\\ \zeta_n^{kh}\,  N x\end{pmatrix}, x \in
\C^{2m}\right\}.
\end{equation}
The proof in the remaining case is completely analogous and the details are left
to the reader.

In the case $k=0$,  we start to observe that, for $z \in E_0^\pm=F_{0,h}^\pm$,  we have
\[
\begin{cases}
\mathcal M\, z= z\\
\mathcal N\,\mathcal M^{h}\, z = \pm z
\end{cases} \iff
\begin{cases}
M\, z\left(t+\dfrac{T}{n}\right)=  z(t)\\
N z(T/n-t) = \pm z(t)
\end{cases}\qquad \forall\, t \in [0,T].
\]
Thus
\[
z(0)=Mz(T/n),Nz(T/n)=\pm z(0)\Rightarrow MNz(0)=\pm z(0), \textrm{ i.e. } z(0)\in
V_\pm(MN) )
\]
and
\[
\pm z\left(\dfrac{T}{2n}\right)= N z\left(\dfrac{T}{2n}\right), \]
and by this it follows that $ z\left(\dfrac{T}{2n}\right)\in V_\pm(N)$.
If $n$ is even, $k=n/2$ and if $z \in F_{n/2,h}^\pm$, then  we have
\[
\begin{cases}
\mathcal M\, z= -z\\
\mathcal N\,\mathcal M^{h}\, z = \pm z
\end{cases} \iff
\begin{cases}
M\, z\left(t+\dfrac{T}{n}\right)=  -z(t)\\
(-1)^h N z(T/n-t) = \pm z(t)
\end{cases}\qquad \forall\, t \in\TT.
\]
Thus  $Mz(T/n)=-z(0),(-1)^hNz(T/n)=\pm z(0)\Rightarrow MNz(0)=(-1)^{(h+1)}z(0)$
and hence   $z(0)\in V_\pm((-1)^{(h+1)} MN)$ as well as $
\pm z\left(\dfrac{T}{2n}\right)=(-1)^h  N z\left(\dfrac{T}{2n}\right)$.
By this it readily follows that
 $ z\left(\dfrac{T}{2n}\right)\in V_\pm((-1)^h N)$. This conclude the proof.
\end{proof}

\paragraph{The quaternionic unitary group}

Let us consider the complex symplectic space  $(\C^{2m},\omega)$ where
$\omega$ is the standard symplectic space defined by
$\omega(x,y)=\langle Jx,y\rangle $. The {\em compact symplectic group $\Sp(m)$\/} is isomorphic to
the group  of unitary and symplectic matrices; i.e.
\[
 \Sp(m)\cong \Sp(2m,\C) \cap \U(2m).
\]
\begin{rem}
 The group $\Sp(m)$ is the subgroup of the invertible quaternionic matrices $\GL(m, \HH)$
 that preserves the Hermitian  form on $\HH^m$
 \[
  \langle x, y \rangle \= \bar x_1 y_1+ \dots \bar x_n y_n.
 \]
We observe that  $\Sp(m)$ is just the {\em quaternionic unitary group\/} $\U(m, \HH)$ and for this reason is
sometimes also termed {\em hyperunitary\/}.
\end{rem}
We assume that  $Q, M \in \Sp(m)$ and $N \in \U(2m)$ be such that the following commutativity properties
holds:
 \begin{equation} \label{eq:relazioni1}
M^n=Q,\ \   N^2 = \Id,\ \   M\, J= J \, M, \ \ N\, J= -J\, N, \ N=N^*, \ \ N \,
M^*= M\, N.
         \end{equation}
\begin{lem}\label{thm:lagrangian}
For any $k=0, \dots, \bar n$, the subspaces $V_\pm(M^kN)$
are Lagrangian subspaces of $(\C^{2m}, \omega)$.
\end{lem}
\begin{proof}
We only prove that $ V_+(N), V_-(N)\in \Lagr(\C^{2n}, \omega)$ and we leave to the  interested
reader the proof that $V_\pm(M^kN)\in \Lagr(\C^{2n}, \omega) $, being completely similar.

Let  $N\in \U(2m)$, $N^2 = \Id,$  $N\, J= -J\, N$ and  $N=N^*$;  so $N^*JN=
-N^*NJ=-J$. For all $x,y \in V_+(N)$,
$N\,x=x$ and $N\,y=y$. Then
\[
\langle J\, x, \,y\rangle=\langle J\,N\,x, N\,y\rangle= \langle N^*\, J\, N\, x,
y\rangle= -\langle  Jx, y\rangle
\]
so $\langle Jx,y\rangle=0$.  This computation immediately shows that  $V_+(N)$ is an isotropic subspace. By the
very same arguments it is possible to conclude also that  $V_-(N)$ is an isotropic subspace. Since
$\C^{2m}=V_+(N)\oplus V_-(N)$, so $V_\pm(N)$ are maximal  isotropic
subspaces and  hence  Lagrangian subspaces.  This conclude the proof.
\end{proof}
We let $ \widehat J=\begin{bmatrix}
             J & 0\\
             0 & J
             \end{bmatrix}$ and we observe that the pair  $(\C^{4m}, \widehat{\omega})$ with $
\widehat{\omega}(x,y)=\langle\widehat{J}x,y\rangle$ is a (complex) symplectic space.   With a
slight abuse of notation, we'll denote by the same symbol $\widehat J$ the operator on $F_k$
induced by $\widehat{J}$.
\begin{lem}\label{thm:relazioni}
The following relations hold
\[
\widehat{M}_k \, \widehat{J} =\widehat{J} \,
\widehat{ M}_k, \ \
(\widehat{N}_k^h)^2=\Id, \ \
\widehat{ N}_k^h\, \widehat{ J}=- \widehat{J}\,
\widehat{ N}_k^h,\ \
\widehat{ N}_k^h = (\widehat{ N}_k^h)^*, \ \
\widehat{N}_k^h\, (\widehat{M}_k^h)^*= \widehat{M}_k
\, \widehat{N}_k^h.
\]
\end{lem}
\begin{proof}
The proof immediately follows by a straightforward calculation.
\end{proof}
Arguing precisely as in Lemma \ref{thm:lagrangian} and as a direct consequence of the relations given in
Equation \eqref{eq:relazioni}, it readily follows that $V_\pm(\widehat M_k^h\widehat N_k ), V_\pm
(\widehat M_k^h) \in \Lagr(\C^{4m}, \widehat{\omega})$, for every  $h=0, \dots, n-1$ and  $k=1,\cdots, \myfloor{(n-1)/2}$.


\section{A dihedral equivariant Bott-type iteration Formulas}\label{sec:special-decomp}

The scope of this section, is  to prove a Bott-type iteration formula for
\begin{itemize}
 \item $\die{n}$-equivariant solutions of a Hamiltonian System under Lagrangian boundary conditions;
 \item $\die{n}$-equivariant solutions of a  Lagrangian System under selfadjoint boundary conditions.
\end{itemize}
The general formula in the case of Hamiltonian systems will be derived in Subsection \ref{subsec:1} whilst
the case of Lagrangian systems will be given in Subsection \ref{subsec:2}.

\subsection{A dihedral-equivariant Bott-type formula for Hamiltonian systems}\label{subsec:1}

Let  $H \in \mathscr C^2\Big([0,T] \times \R^{2m},\R\Big)$ be a  time-dependent
Hamiltonian function
and let $L$ be a Lagrangian subspace of the symplectic space $(\R^{2m}\oplus
\R^{2m}, -\omega \oplus \omega)$.
We define the closed (in $L^2$) subspace
$\mathcal D(T, L)\= \Set{z\in W^{1,2}([0,T], \R^{2m})| \big(z(0), z(T)\big) \in
L}$. We  denote by $\overline{\mathcal D(T,L)}$ the closure in the $W^{1/2,2}$-norm
topology of
$\mathcal D(T,L)$ and we consider the {\em symplectic action functional\/}
\begin{equation}\label{eq:symplectic-action}
\mathscr A_H: \overline{\mathcal D(T,L)}\to \R \textrm{ defined by }
\mathscr A_H(z)\=\int_0^T \left[\Big\langle -J\, \dfrac{dz(t)}{dt},
z(t)\Big\rangle- H\big(t, z(t)\big)\right]\, dt.
\end{equation}
By standard regularity arguments, it follows that a critical point of $\mathscr A_H$ is
weak (in the Sobolev sense)-solution of the boundary value problem
\begin{equation}\label{eq:ham-sys-1}
\begin{cases}
 \dot z(t)= J\nabla H\big(t, z(t) \big),  \qquad t \in [0,T]\\
 \big(z(0), z(T)\big) \in L.
\end{cases}
\end{equation}
\begin{rem}
We observe that the periodic solutions can be obtained by setting $L=\Delta$
where $\Delta$ denotes the diagonal subspace
in the product space $\R^{2m} \oplus \R^{2m}$.
\end{rem}
Let $z$ be a solution of the Hamiltonian System given in Equation
\eqref{eq:ham-sys-1} and let us denote by $\gamma$ the fundamental solution of
its linearisation along $z$; namely $\gamma$ solves
\begin{equation}\label{eq:ham-sys-1-lin}
\begin{cases}
 \dot\gamma(t)=J D^2H\big(t, z(t)\big)\, \gamma(t),  \qquad t \in [0,T]\\
 \gamma(0)=\Id_{2m}
\end{cases}.
\end{equation}
We set  $B(t)\=D^2H\big(t, z(t)\big)$ and let us define the
closed selfadjoint Fredholm operators in $L^2$ having  domain
$\mathcal D(T, L)\= \Set{z\in W^{1,2}([0,T], \R^{2m})| \big(z(0), z(T)\big) \in
L}$ and given by
\begin{equation}\label{eq:a0ea1}
\mathcal A_1\=- J\dfrac{d}{dt} - B(t) \textrm{ and }
\mathcal A_0\=- J\dfrac{d}{dt}.
\end{equation}
Following authors in \cite[Definition 2.1]{HS09},
we define the {\em relative Morse index\/} of $z$ as follows
\begin{equation}\label{eq:relative-Morse-index}
 \iRel(z)\= \irel\left(\mathcal A_0, \mathcal A_1\right)= -\spfl\left(\mathcal
A;[0,1]\right)
\end{equation}
where $\mathcal A:[0,1] \to \cfsa(E)$ is a continuous path of closed
self-adjoint Fredholm operators defined
by  $\mathcal A_s\= \mathcal A_0 + B_s$ where
$s\mapsto B(s)$ is such that $ B_0= 0$ and $B_1\=B$
on the $s$-independent domain $\mathcal D(T,L)$. We define the {\em geometrical
index\/} of the solution $z$ of the Hamiltonian system given inEquation \eqref{eq:ham-sys-1} as
\begin{equation}\label{eq:Maslov-z}
 \iMas(z)\= \iCLM\big(L, \Graph(\gamma);[0,T]\big).
\end{equation}
We observe that $z_s \in \ker \left(\mathcal A(s)|_{\mathcal D(T, L)}\right)$ if
and only if $z_s$ is  a solution of the
linear Hamiltonian boundary value problem
\begin{equation}\label{eq:ham-sys-family}
\begin{cases}
 \dot z_s= J\, B_s(t)\, z_s(t),  \qquad t \in [0,T]\\
 \big(z_s(0), z_s(T)\big) \in L \cap \Graph\big(\gamma_s(T)\big)
\end{cases}
\end{equation}
where $\gamma_s$ is the fundamental solution of the Hamiltonian system given in Equation
\eqref{eq:ham-sys-family}.
\begin{prop}({\bf A spectral flow formula \/})\label{prop1.12}  Under the
previous notations,  we have
\[
 \iMas(z)= \iRel(z).
\]
\end{prop}
\begin{proof}
For the proof of this result we refer the interested reader to \cite[Theorem
2.5]{HS09}.
\end{proof}
\begin{rem}
 It is worth noticing that if $L=L_1 \oplus L_2 \in \Lagr(\R^{2m}\oplus \R^{2m},
-\omega \oplus \omega)$,
 where $L_i \in \Lagr(\R^{2m},\omega)$, for $i=1,2$, then we  have
 $\iCLM\big(L_1 \oplus L_2, \Graph(\gamma); [0,T]\big)= \iCLM\big(L_2,
\ell_1;[0,T]\big)$ where $\ell_1(\cdot)\=
 \gamma(\cdot)\, L_1$.
\end{rem}

Let us  consider the operator $\mathcal N \, \mathcal M^{n-1}$ on $E$ which is
given by
$(\mathcal N \, \mathcal M^{n-1}\, z)(t)\= N\, M^{n-1}\, z(T-t)$. By taking into
account Lemma \ref{thm:spectrumSRh}, we can decompose $E$ into the following orthogonal
direct sum
\[
 E= V_+(\mathcal N \, \mathcal M^{n-1})\oplus V_-(\mathcal N \, \mathcal M^{n-1})
\]
By a  direct calculation it follows
that if  $z \in  V_+(\mathcal N \, \mathcal M^{n-1})$, then $z(t)=N\,M^{n-1}\, z(T-t)$ for all $t \in [0,T]$
and hence  $z(0)= N\, M^{n-1}\, z(T)$. By  multiplying the last equation on the left
by $M\,N$, we
get  $M\, N\, z(0)= M^m\, z(T)= Q \, z(T)\,= z(0)$, namely $z(0) \in V_+(M\,N)$.
Analogously, $z(T/2)= N\, M^{n-1}\, z(T/2)$, or which is equivalent
$z(T/2) \in V_+(N\, M^{n-1})$. In conclusion $V_\pm$ are the closed
$\die{n}$-invariant subspaces defined as
\begin{equation}
 V_\pm(\mathcal N \, \mathcal M^{n-1})\=\Set{z \in E| N  M^{n-1} z=\pm z, \quad
 \big(z(0), z(T/2)\big) \in V_\pm(M\,N)\times
 V_\pm(N \, M^{n-1})}.
\end{equation}
For $k=1, \dots, \myfloor{(n-1)/2}$, we denote by $F_k^\pm$ the closed
$\die{n}$-invariant subspaces defined by
\begin{equation}
\begin{split}
 F_k^\pm(\mathcal N \, \mathcal M^{n-1})\=\left\{u \in F_k\Big|\begin{pmatrix} 0
& \mathcal N\,\mathcal M^{n-1}\\
 \mathcal N\,\mathcal M^{n-1} & 0\end{pmatrix}\,u = \pm u, \quad u(0) \in V_\pm
 \begin{pmatrix} 0 & M\,N \\
 M\,N & 0\end{pmatrix}, \right.\\ \left. \ u\left(\dfrac{T}{2n}\right) \in V_\pm
 \begin{pmatrix} 0 & \zeta_n^{-k(n-1)}\, N \\
 \zeta_n^{k(n-1)}\,N & 0\end{pmatrix}
\right\}
\end{split}
\end{equation}
and we set
\[
 E_0^\pm(\mathcal N\, \mathcal M^{n-1})\=\Set{z \in E_0|\mathcal N\,
 \mathcal M^{n-1}\, u= \pm u, \ \ \big(u(0), u(T/(2n)\big) \in
 V_\pm(M\,N)\times V_\pm(N)}  .
 \]

By this,  the following orthogonal direct sum decomposition holds
\begin{equation}\label{eq:direct-sum-decomp}
 E^\pm(\mathcal N \, \mathcal M^{n-1}) = \bigoplus_{k=0}^{\bar n}
F^\pm_k(\mathcal N \, \mathcal M^{n-1})
\end{equation}
where$F^\pm_0 = E_0$ and, for $n$ even $ F^\pm_{n/2}=
E^\pm_{n/2} $ where
\begin{equation}
 E_{n/2}^\pm\=\Set{z \in E_{n/2}|\mathcal N\, \mathcal M^{n-1}\, u= \pm u, \ \
\big(u(0), u(T/(2n)\big) \in V_\pm(M\,N)\times V_\mp(N)}.
\end{equation}

\paragraph{Proof of Theorem \ref{thm:Bott-diedrale-Hamiltoniano-intro}.}
We start  define  the analytic path
$\lambda \mapsto \mathcal A(\lambda)\=\mathcal A_0-\lambda B$ where $\mathcal A_0$ and $B$ were
defined in Equation \eqref{eq:a0ea1}. Since both these operators commutes with
 $\mathcal M$ and $\mathcal N$, then the whole path $\mathcal A(\lambda)$ too.  In order to conclude the proof it is
enough to invoke Proposition \ref{thm:cor2.7} and  Proposition  \ref{prop1.12} and the symplectic additivity properties of the
geometrical indices. This conclude the proof. \qed

\subsection{A dihedral-equivariant Bott-type formula for Lagrangian
systems}\label{subsec:2}

Let $L \in \mathscr C^2\big([0,T]\times  \R^{2m}, \R\big)$ be a  Lagrangian
function  and let $\mathscr S_L: W^{1,2}([0,T]; \R^m)\to \R $ be the {\em
Lagrangian action functional\/} defined as
\begin{equation}\label{eq:lagr-action}
 \mathscr S_L(x)\=\int_0^T L\big(t, x(t), \dot x(t)\big)\, dt.
\end{equation}
We assume that the function $L$ satisfying the {\em Legendre convexity
condition\/}:
\begin{equation}
 \Big \langle D_{vv}^2\,L(t,q,v)\, w, w \Big \rangle>0  \textrm{ for } t \in
[0,T], \
 w \in \R^m\setminus \Set{0}, \ (q,v) \in \R^m \times \R^m
\end{equation}
and let $V\subset \R^m \oplus \R^m$ be  a  fixed subspace. A weak (in a Sobolev sense) solution of the
Euler-Lagrange Equation with the
boundary condition $\big(\gamma(0), \gamma(T)\big) \in V$ is  a critical point
of $\mathscr S_L$ in the
space
\begin{equation}
 E_V\=\Set{x\in W^{1,2}([0,T], \R^m)| \big(x(0), x(T)\big) \in
V}.
\end{equation}
More precisely, $x$ is  a solution of the following second order system
\begin{equation}\label{eq:EL}
 \begin{cases}
  \dfrac{d}{dt}\partial_v\,L\big(t,x(t), \dot x(t)\big)) -\partial_q\, L\big(t,
x(t), \dot
  x(t)\big)=0, \qquad t \in [0,T]\\
  \big(x(0), x(T)\big) \in V, \ \ \left(D_v L\big(0,x(0), \dot x(0)\big),
  -D_v L\big(T,x(T), \dot x(T)\big) \right)\in V^\perp,
 \end{cases}
\end{equation}
where $V^\perp$ is the orthogonal complement of $V$ in $\R^m \oplus \R^m$.
By using the Legendre transformation $p= D_v L(t, q,v)$ and setting $H(t,p,q)=
\langle p, v\rangle -
L(t,q, v)$ the Euler-Lagrange equation given in Equation \eqref{eq:EL} can be converted into
the following Hamiltonian System
\begin{equation}\label{eq:ham-indotto}
 \dot z(t)= J\, \nabla H\big(t, z(t)\big)
\end{equation}
with $z(t)\=\big(y(t), x(t)\big)=\big(D_v L(t,x(t),y(t)),x(t)\big)$. (Cf., for instance to  \cite{APS08} for further details). We let
$\widehat J= -J \oplus J$ and we observe
that the subspace $L_V \=J V^\perp \oplus V\subset \R^{2m}\oplus \R^{2m}$ is a
Lagrangian subspace of $(\R^{2m}\oplus \R^{2m}, -\omega \oplus \omega)$.
Thus $x \in \mathscr C^2([0,T] \times \R^{m}, \R)$ solves the boundary value
problem given in Equation \eqref{eq:EL}
if and only if $z \in \mathscr C^2([0,T] \times \R^{2m}, \R)$ is a solution of
the Equation \eqref{eq:ham-indotto} under
the following Lagrangian  boundary condition
\[
 \big(z(0),z(T)\big)\in L_V.
\]
\begin{prop}({\bf Morse-type Index Theorem\/})\label{thm:morse-index-theorem} Let $x$ be a critical point of
$\mathscr S_L$ and we assume that the Legendre convexity condition holds.
Then the Morse index of $x$ is finite and we have
\[
\iMor(x) + \iota(L)= \iRel(z)=\iMas(z)
\]
where $\iota(L)=\irel\left(- J\dfrac{d}{dt}, - J \dfrac{d}{dt}
-C\right)$ and
$ C$ is the operator pointwise induced by the matrix
\[
C=\begin{bmatrix} \Id_m & 0 \\0 & -\Id_m\end{bmatrix}
\]
on the domain $\mathcal D(T,L)$.
\end{prop}
\begin{proof}
 For the proof of this result we refer the interested reader to \cite[Theorem
3.4]{HS09}.
\end{proof}
\begin{rem}
 We observe that for Dirichlet boundary condition, the corresponding Lagrangian
subspace is the
horizontal  Lagrangian subspace $L_D\=\R^{m} \oplus \{0\} $ of the phase space and
in this case it is easy to compute that
 $\iota(L_D)=m$. We observe that in the case of periodic boundary condition the
corresponding Lagrangian subspace is the
 diagonal $\Delta$ and in this case $\iota(\Delta)=m$. By a direct calculation it
is possible to check that in the
 case of Neumann boundary condition the corresponding Lagrangian subspace $L_N$
in  the phase space is the
 horizontal and $\iota(L_N)=0$. For further details we refer the interested reader
to \cite[Remark 3.6]{HS09}.
\end{rem}
\paragraph{Proof of Theorem \ref{thm:Bott-diedrale-Lagrangiano-intro}.}
By a direct linearisation of the Euler-Lagrangian equation along the solution  $x$,
we get
\begin{equation}\label{eq:sl}
-\dfrac{d}{dt}(P(t)\dot{y}+Q(t)y)+Q^T(t)\dot{y}+R(t)y=0 \qquad t \in [0,T].
\end{equation}
We observe that  $y$ is solution of Equation \eqref{eq:sl} if and only if $y\in\ker
(\mathcal {I})$; moreover the associated
linear Hamiltonian system is given by
\begin{equation}\label{eq:n2.8}
\dot z(t) = JB(t)z(t), \qquad t \in [0,T]
\end{equation}
where
\begin{equation}
B(t)\=
\begin{bmatrix}
P^{-1}(t) &
-P^{-1}(t) Q(t)\cr -Q^T(t) P^{-1}(t) &
Q^T(t)P^{-1}(t)Q(t)-R(t)
\end{bmatrix}
\end{equation}
Let $S_d\=\begin{bmatrix}S & 0 \\0 &S \end{bmatrix}$ and  $N_d\=\begin{bmatrix}-N & 0
\\0 &N \end{bmatrix}$. Then we have
\[
B(t)S_d=S_dB(t+\dfrac{T}{n}) \textrm{ and }  B(t)N_d=N_dB(\dfrac{T}{n}-t).
\]
Let us define the unitary operators pointwise given by
\[
\widehat{g}_1x\=S_dx(t+T/n) \quad \textrm{ and } \quad
\widehat{g}_2x=N_dx(T/n-t)
\]
and we observe that  $-J \dfrac{d}{dt}$ as well as  $ B$ both  commute with $\widehat{g}_1$ and
$ \widehat{g}_2$.

By invoking Proposition \ref{thm:morse-index-theorem},  we  get
the relation of Maslov index and Morse index
\[
\iindex{x} + \iota(L)= \iRel(z)=\iMas(z)
\]
where $\iota(L)=\irel\left(- J\dfrac{d}{dt}, - J \dfrac{d}{dt}
-C\right)$ and
$C=\begin{bmatrix} \Id_m & 0 \\0 & -\Id_m\end{bmatrix}$. Since the operator  pointwise induced by
$C$ commute with $\widehat{g}_1, \widehat{g}_2$, then get
\[
\irel\left(- J\dfrac{d}{dt}, - J \dfrac{d}{dt}
-C\right)=\irel\left(- J\dfrac{d}{dt}\Big|_{E^+}, - J \dfrac{d}{dt}
-C\Big|_{E^+}\right) +\irel\left(- J\dfrac{d}{dt}\Big|_{E^-}, - J
\dfrac{d}{dt} -C\Big|_{E^-}\right)    
\] and
\[
\irel\left(- J\dfrac{d}{dt}\Big|_{E^\pm}, - J \dfrac{d}{dt}
-C\Big|_{E^\pm}\right)=\sum_{k=0}^{\bar n}\irel\left(-
J\dfrac{d}{dt}\Big|_{F^\pm_{k,h}}, -J \dfrac{d}{dt}
-C\Big|_{F^\pm_{k,h}}\right).
\]
The proof readily follows by invoking Theorem \ref{thm:Bott-diedrale-Hamiltoniano-intro}.
This conclude the proof. \qed


\section{Some dynamical and variational consequences}\label{sec:applications}

The aims of this section is to derive some dynamical  consequences of
the theory developed in the previous sections. More precisely, in Subsection
\ref{subsec:special-decomp-2}, we apply the dihedral equivariant Bott-type formula to the
celebrated figure-eight orbit for the planar three-body problem whilst Subsection
\ref{subsec:ultima} we investigate the hyperbolicity and the strongly instability of
symmetric Lagrangian systems.

\subsection{Decomposition of the  figure-eight
orbit}\label{subsec:special-decomp-2}

We start to  consider three point particles in the Euclidean plane, namely
$(x_1, x_2 ,x_3) \in (\R^2)^3$ self-interacting with the Newtonian
gravitational potential $U$ and having unitary mass.
We introduce the Jacobian coordinates through the  canonical transformation
$
\begin{bmatrix} v\\u\end{bmatrix}= \begin{bmatrix}
\traspm{K} &0\\0&K\\\end{bmatrix} \begin{bmatrix}
y\\x\end{bmatrix},
$
where $\traspm{K} $  denotes the transpose of $K^{-1}$, $
u=(u_{1},u_{2},u_{3})^{T},v=\trasp{(v_{1},v_{2},v_{3})}\in(\R^{2})^{3}$ and
\[
K=\begin{bmatrix}
0_2&-\dfrac{1}{\sqrt{2}}I_{2}&\dfrac{1}{\sqrt{2}}I_{2}&\\
\dfrac{\sqrt{2}}{\sqrt{3}}I_{2}&-\dfrac{1}{\sqrt{6}}I_{2}&-\dfrac{1}{\sqrt{6}}I_{2}
&\\
\dfrac{1}{3}I_{2}&\dfrac{1}{3}I_{2}&\dfrac{1}{3}I_{2}\end{bmatrix}
\]
and by a straightforward calculation, we readily get
 $ K^{-1}= \begin{bmatrix}
0_2&\dfrac{\sqrt{6}}{3}I_{2}&I_{2}&\\
-\dfrac{\sqrt{2}}{2}I_{2}&-\dfrac{\sqrt{6}}{6}I_{2}&I_{2}&\\
\dfrac{\sqrt{2}}{2}I_{2}&-\dfrac{\sqrt{6}}{6}I_{2}&I_{2}\end{bmatrix}$. By a direct computation,
with respect to the Jacobian coordinates the configuration space transforms into the following
\[
\widetilde{\mathcal X}\=\Set{u=(u_{1},u_{2},u_3)\in(\R^{2})^{3}| u_3=0}.
\]
Similarly the collision set fits into the following
\[
\widetilde{\Delta}\=\Set{u\in\widetilde{\mathcal X}|u_{1}=\sqrt{3}u_{2} \;or\;
u_{1}=-\sqrt{3}u_{2}  \;or\;  u_{1}=0}.
\]
Then the transformed configuration space is
$\overline{\mathcal X}\=\widetilde{\mathcal X}\setminus \widetilde{\Delta}.$
In these new coordinates $ u=(u_{1},u_{2})$,  if we denote by
$\widetilde{g_{1}},\widetilde{g_{2}}$ the generators of the group acting on
$W^{1,2}(\R/(T\Z), \overline{\mathcal X})$,then
\begin{equation}\label{eq:azione}
(\widetilde{g}_{1}\circ u)(t)=\widetilde{S}u(t+\dfrac{T}{6}) \textrm{ and }
(\widetilde{g}_{2}\circ u)(t)=\widetilde{N}u(\dfrac{T}{6}-t).
\end{equation}
 Here $ \widetilde{S}=\begin{bmatrix}
\dfrac{1}{2}R_{2}&\dfrac{\sqrt{3}}{2}R_{2}\\-\dfrac{\sqrt{3}}{2}R_{2}&\dfrac{1}{2}R_
{2}\end{bmatrix}$ and
 $ \widetilde{N}=\begin{bmatrix} -R_{2}&0\\0&R_{2}\end{bmatrix}$.    It's easy
to check that $\widetilde{S}$ and $\widetilde{N}$ satisfy the following properties
  \begin{enumerate}
  \item $ \widetilde{S}\in \OO(4)$ and
  $ \widetilde{S}^{6}=I $
  \item $ \widetilde{N}\in \OO(4),\widetilde{N}=\widetilde{N}^{T}$ and
$\widetilde{N}^{2}=I $
\item
  $\widetilde{N}\widetilde{S}^{T}=\widetilde{S}\widetilde{N}.$
  \end{enumerate}
 \indent Let $E=W^{1,2}(\R/(T\Z),\mathbf{C}^{4}).$ The eigenvalues
of $\widetilde{g}_{1}$ are
$\omega_{i}=e^{i\pi\sqrt{-1}/3}$ for $i=0,1,\cdots,5.$ Denoting by  $ E_{i} $
the eigenspace corresponding to $\omega_{i}$, then we have
 \[
 E=\bigoplus\limits_{i=0}^{5}E_{i}.
 \]
 By taking into account the properties of  the matrices $ \widetilde{S}$ and
 $ \widetilde{N},$ it follows that $\widetilde{g}_{1}, \widetilde{g}_{2}$ defined in
 Equation \eqref{eq:azione} defined an action of the  $\die{6}$-group on $E$.
  Now, for any $u\in E_{i},$ we have $ \widetilde{g}_{1}u=\omega_{i}u$. Moreover, since
\[
 \widetilde{g}_{1}\widetilde{g}_{2}u=\widetilde{g}_{2}\widetilde{g}^{-1}_{1}
u=\overline{\omega}_{i}\widetilde{g}_{2}u,
\]
so, we get
$\widetilde{g}_{2}u\in\bar{E_{i}}$ and hence  for any $ h=0,1\dots,5 $
\[
\widetilde{g}_{2}\widetilde{g}_{1}^{h}:E_{i}\longrightarrow \bar{E_{i}},
\]
where we denoted by $ \bar{E_{i}}$ the eigenspace corresponding to the eigenvalue $
\bar{\omega_{i}}$.
We now define the following subspaces
\[
F_{0}=E_{0}, \quad F_{3}=E_{3} \quad \textrm{ and finally }F_{i}=E_{i}\bigoplus E_{6-i},i=1,2
\]
and we observe that
\begin{equation}
\begin{bmatrix}
0&\widetilde{g}_{2}\widetilde{g}_{1}^{h}\\\widetilde{g}_{2}\widetilde{g}_{1}^{h}
&0
\end{bmatrix}
F_{i}=F_{i} \textrm{ for } i=1,2\quad  \textrm{ and }\quad
 \widetilde{g}_{2}\widetilde{g}_{1}^{h}F_{i}=F_{i} \textrm{ for }i=0,3.
 \end{equation}
By a direct calculation, for $i=1,2$ we get
\[
F_{i}=\Set{\left( \begin{matrix}u_{1}\\u_{2}\\\end{matrix}\right)\in
W^{1,2}([0,T/6],\mathbf{C}^{4}\bigoplus \mathbf{C}^{4}), \left(
\begin{bmatrix}\bar{\omega_{i}}\widetilde{S}&0\\0&\omega_{i}\widetilde{S}
\\\end{bmatrix}\right)
 \left(
\begin{bmatrix}u_{1}(T/6)\\u_{2}(T/6)\\\end{bmatrix}\right)=
   \begin{bmatrix}u_{1}(0)\\u_{2}(0)\\\end{bmatrix}}
\]
and $F_{i}=\Set{u\in
W^{1,2}([0,T/6],\mathbf{C}^{4}),u(0)=(-1)^{i}u(T/6),}$ for $i=0, 3$.
From Lemma \ref{thm:eigenspaces-Lagrangians}, we also have
 \begin{multline}
 F_{i,h}^{\pm}=\left\{
\begin{bmatrix}u_{1}\\u_{2}\\\end{bmatrix}\in
W^{1,2}([0,T/12],\C^{4}\bigoplus\C^{4})\Big|
\begin{bmatrix}u_{1}(0)\\u_{2}(0)\\\end{bmatrix}\in V_{\pm}
\begin{bmatrix}0&\overline{\omega}_{i}^{h+1}\widetilde{S}
\widetilde{N}\\\omega_{i}^{h+1}\widetilde{S}\widetilde{N}&0\\\end{bmatrix} \textrm{ and }\right.\\
\left.\begin{bmatrix}u_{1}(T/12)\\u_{2}(T/12)\\\end{bmatrix}\in
V_{\pm}
\begin{bmatrix}0&\overline{\omega}_{i}^{h}\widetilde{N}\\\omega_{i}^{h}\widetilde
{N}&0\\\end{bmatrix}\right\}
\end{multline}
\begin{multline}
F_{0,h}^{\pm}=\Big\{ u\in W^{1,2}([0,T/12],\mathbf{C}^{4}),u(0)\in
V_{\pm}(\widetilde{S}\widetilde{N}),u(T/12)\in
V_{\pm}(\widetilde{N})\Big\}\\
F_{3,h}^{\pm}=\Big\{ u\in W^{1,2}([0,T/12],\mathbf{C}^{4}),u(0)\in
V_{\mp}((-1)^{h}\widetilde{S}\widetilde{N}),u(T/12)\in
V_{\pm}((-1)^{h}\widetilde{N})\Big\}.
\end{multline}
 So, $E$ can be decomposed into a direct sum of orthogonal $\die{6}$-invariant closed subspaces
\[
 E=E_{h}^{+}\bigoplus E_{h}^{-}
 \]
where
\[
E_{h}^{+}\=\bigoplus\limits_{k=1}^{2}F_{k,h}^{+}\oplus
F_{0,h}^{+}\oplus F_{3,h}^{+} \quad \textrm{ and } \quad
E_{h}^{-}=\bigoplus\limits_{k=1}^{2}F_{k,h}^{-}\oplus F_{0,h}^{-}\oplus
F_{3,h}^{-}.
\]
Thus the Morse index $\iMor(x)$ of the Figure-eight orbit can be
expressed as:
\[
\iMor(x)= \sum \limits_{k=0}^{3} \Big[\iMor(F_{k,h}^{+})+\iMor(F_{k,h}^{-})\Big].
\]
Since the  figure-eight orbit is a  minimizer on the $\die{6}$-equivariant loop space of
the collision manifold, this in particular implies that $\iMor(F_{0,0}^{+})=0$.

With a slight abuse of notation we denote by  $\Z_2$, $\Z_3$
the group generated by $\widetilde{g}_1^3,
\widetilde{g}_1^2$, respectively (actually we are identifying the abstract cyclic groups $\Z_2$ and $\Z_3$
with their unitary representation on $E$) and we set $\iMorh{2}$ and  $\iMorh{3}$  be the Morse indices of
the figure eight orbit as the critical point of the action functional restricted on the spaces fixed by the
action of $\Z_2$ and  $\Z_3$.
Since  the $\Z_2$ fixed space is  $F_0\oplus F_2$,
and the  $\Z_3$ fixed space is $F_0\oplus F_3$, we get
\[
\iMorh{2}(x)=\iMor(F_0)+\iMor(F_2) \quad \textrm{ and } \quad   \iMorh{3}(x)=\iMor(F_0)+\iMor(F_3).
\]
By invoking \cite[Remark 5.12]{HS09} we already know that
  $\iMor(x)=2$, $\iMorh{2}=\iMorh{3}=0$. Thus, we get
  \[
  \iMor(F^\pm_{0,h})= \iMor(F^\pm_{3,h}) =\iMor(F^\pm_{2,h})=\iMor(E_3)=\iMor(E_2)=\iMor(E_4)=0 \qquad \textrm{ for }
  h=0,\cdots,5
     \]
and
    \[
    \iMor(F^\pm_{1,h})=\iMor(E_1)=\iMor(E_5)=1.
    \]
\paragraph{Proof of Theorem \ref{thm:main-3bp}.} The proof of this result readily follows by
summing up all the previous computations performed in Subsection \ref{subsec:special-decomp-2}. \qed


\subsection{Hyperbolicity and strong instability of Lagrangian systems}\label{subsec:ultima}

The aim of this paragraph is to prove some instability results.
To do so we briefly recall the definition of {\em splitting numbers\/} and we fix our notations.

Let $L \in \mathscr C^2([0,T] \times \R^{2m}, \R)$  be a Lagrangian function satisfying the Legendre convexity condition and, as before, we
consider the Lagrangian action functional
\begin{equation}
\mathscr S_L:E_\R \to \R \textrm{ defined by } \mathscr S_L(x)=\int_0^T L\big(t, x(t), \dot x(t)\big)\, dt,
\end{equation}
where $E_\R\=W^{1,2}([0,T]; \R^m)$. We assume that $x$ is a critical point; by the second variation of $\mathscr S_L$,
we get the associated index form given by
\begin{equation}
\mathcal {I}_{\R}(\xi,\eta)=\int_0^{T}\Big[\langle(P\dot{\xi}+Q\xi),\dot{\eta}\rangle+\langle \trasp{Q}\dot{\xi},
\eta\rangle+\langle R\xi, \eta\rangle\Big]dt \qquad  \xi,\eta\in E.
\end{equation}
For any $\omega \in \U$, we let
\[
E^{1}(\omega)\=\Set{u \in E | u(0)=\omega\, u(T)}
\]
where we denoted by $E$  the complexification of $E_\R$; i.e. $E\=E_\R \otimes \C$. The corresponding $\omega$-index
form is given by
\begin{equation}\label{eq:index-form-complex}
\mathcal {I}_\omega(\xi,\eta)=\int_0^{T}\Big[\langle(P\dot{\xi}+Q\xi),\dot{\eta}\rangle+\langle \trasp{Q}\dot{\xi},
\eta\rangle+\langle R\xi,  \eta\rangle \Big]dt \qquad   \xi,\eta\in E
\end{equation}
where we have extended the Euclidean product  $\langle, \cdot,\cdot \rangle$ to the standard Hermitian product.
In what follows we denote by $\iMor(\omega,x)$ the Morse
index of $\mathcal I_\omega$ on $E^1(\omega)$.
\begin{defn}
The {\em splitting number of $x$ at $\omega \in \U$\/} is given by
\begin{equation}\label{eq:S-N}
\mathcal{S}^\pm(\omega, x)=\lim_{\theta\rightarrow\pm0}\Big[\iMor(\omega e^{\sqrt{-1}\theta},x)-\iMor(\omega,x)\Big].
\end{equation}
\end{defn}
Integrating by parts in Equation \eqref{eq:index-form-complex}, we get that linear Sturm system
\begin{equation}\label{eq:MS-system}
\mathcal A \xi\=-\dfrac{d}{d t} \big(P(t) \dot \xi + Q(t) \xi\big) + \trasp{Q}(t)\dot \xi + R(t) \xi=0 \qquad t \in [0,T]
\end{equation}
which reduces to the linear Hamiltonian system
\begin{equation}\label{eq:hamsys}
\dot z = J B(t) z, \qquad t \in [0,T]
\end{equation}
where
\[
B(t)\= \begin{bmatrix}
P^{-1}(t) & - P^{-1}(t) Q(t)\\
-\trasp{Q}(t)P^{-1}(t) & \trasp{Q}(t) P^{-1}(t)Q(t)-R(t)
\end{bmatrix}.
\]
Let $\gamma$ be the fundamental solution of the Hamiltonian system given in Equation \eqref{eq:hamsys} and
let $M_{T}\=\gamma(T)$ be the induced monodromy matrix.
\begin{lem}\label{rel:splitting-numbers}
Under the previous notation, we have
\[
\mathcal S^{\pm}(\omega,x)= S^{\pm}_{M_{T}}(\omega), \qquad \omega \in \U,
\]
where $S^{\pm}_{M_{T}}(\omega)$ denotes the splitting numbers defined in \cite[pag. 191, Equation (4)]{Lon02}.
\end{lem}
\begin{proof}
For the proof of this result we refer the interested reader to \cite[pag.252, Corollary 4]{Lon02} or \cite{HPY17-geo} and references therein. This conclude the proof.
\end{proof}
\begin{defn}
The Morse-Sturm system given in Equation \eqref{eq:MS-system}, is termed {\em index hyperbolic\/} if
\begin{equation}
\Big( S_{M_{T}}^+(e^{i2\pi \theta}), S_{M_{T}}^-(e^{i2\pi \theta})\Big)=
\begin{cases}
(0,0) & \textrm{ if } \theta \in \Q\\
(p,p) & \textrm{ if } \theta \notin \Q
\end{cases}
\end{equation}
\end{defn}
It is worth to observe that $\mathcal A$ is degenerate on
\begin{equation}
E^{2}(\omega)\=\Set{u \in W^{2,2}([0,T], \C^{m})|u(0)=\omega u(T) \textrm{ and } \dot u(0)=\omega \dot u(T)}
\end{equation}
if and only if $\omega$ is an eigenvalue of $M_{T}$; i.e. in symbols $\omega \in \mathfrak{sp}\big(M_{T}\big)$.

The next result gives a sufficient condition in terms of the linearized operator $\mathcal A$ in order that a $\die{n}$-invariant
critical point  is index hyperbolic.
\begin{thm}\label{thm:1-ultimasezione}
Let $S, N \in \OO(m)$ be such that $S^{n}= \Id_{m}$ and $N^{2}=(NS)^{2}=\Id_{m}$ acting on $E_\R$ dihedrally through  the
action given by
\begin{multline}\label{eq:action-2-intro}
\mathcal S: E \ni  u\longmapsto  (\mathcal S\, u)(\cdot) \=S\, u\left(\cdot+ T/n\right)\in E_\R
\textrm { and }\\
\mathcal N: E \ni  u\longmapsto  (\mathcal N\, z)(\cdot) \=N\, u\left(T/n-t\right)\in E_\R.
\end{multline}
Under the above notation,  if the restriction of $\mathcal A$ given in Equation \eqref{eq:MS-system-intro} on
\begin{equation}
E^2\left(\dfrac{T}{2n}\right)\=\Set{u \in W^{2,2}\left(\left[0,\dfrac{T}{2n}\right], \C^{m}\right)| \dot{u}(0)=\dot{u}\left(\dfrac{T}{2n}\right)=0}
\end{equation}
is positive semi-definite, then $x$ is index hyperbolic.
\end{thm}
\begin{proof}
We start to observe that by the group action $(\mathcal N u)(t)=Nu(T/n-t)$ on $E^2(T/n)$,  we have the following decomposition
\[
E^2\left(\dfrac{T}{n}\right)=E^2_+\left(\dfrac{T}{2n}\right)\oplus E^2_-\left(\dfrac{T}{2n}\right)
\]
where
\begin{multline}
E^2_+\left(\dfrac{T}{2n}\right)\= \Set{u \in W^{2,2}\left(\left[0,\dfrac{T}{2n}\right], \C^{m}\right)|
u\left(\dfrac{T}{2n}\right) \in V_+(N),\ \dot{u}(0)=0,\ \dot{u}\left(\dfrac{T}{2n}\right) \in V_-(N)}\\
E^2_-\left(\dfrac{T}{2n}\right)\= \Set{u \in W^{2,2}\left(\left[0,\dfrac{T}{2n}\right], \C^{m}\right)|
u\left(\dfrac{T}{2n}\right) \in V_-(N),\ \dot{u}(0)=0,\ \dot{u}\left(\dfrac{T}{2n}\right) \in V_+(N)}.
\end{multline}
Since the spectrum of $\mathcal N$ is $\{+1,-1\}$, if $u \in V_{1}(\mathcal N)$, then we have
\begin{multline}
(\mathcal N u)(t)= N u \left(\dfrac{T}{n}-t \right)=u(t) \Rightarrow Nu \left(\dfrac{T}{n}\right) =
u(0) \Rightarrow u(0)\in \C^m \textrm{ and }\\
(\mathcal N u)\left(\dfrac{T}{2n}\right)= N u \left(\dfrac{T}{n}-\dfrac{T}{2n}\right)=
u\left(\dfrac{T}{2n}\right) \Rightarrow u \left(\dfrac{T}{2n}\right) \in V_+(N).
\end{multline}
Analogously, we have
\begin{multline}
(\mathcal N u)(t)= N u \left(\dfrac{T}{n}-t \right)=u(t) \Rightarrow -N\dot u \left(\dfrac{T}{n}-t\right) =
\dot u(t) \Rightarrow  -N\dot u \left(\dfrac{T}{n}\right) =
\dot u(0) \Rightarrow \dot u(0)=0\\
-N\dot u \left(\dfrac{T}{2n}\right) =
\dot u\left(\dfrac{T}{2n}\right)  \Rightarrow \dot u \left(\dfrac{T}{2n}\right) \in V_-(N).
\end{multline}
By this calculation, we get the definition of $ E^2_+\left(\dfrac{T}{2n}\right)$ in which we omit the
trivial condition $u(0)\in \C^m$. If  $u \in V_{-1}(\mathcal N)$, by the very same computation we
get  the subspace  $ E^2_-\left(\dfrac{T}{2n}\right)$.

Since $\mathcal A$ is positive semi-definite 
on $E^2\left(\dfrac{T}{2n}\right)$, then the index form $\mathcal{I}$ is positive semi-definite  on 
$E^1\left(\dfrac{T}{2n}\right)=W^{1,2}\left(\left [0,\left(\dfrac{T}{2n}\right)\right],\C^m\right)$. Clearly, we have
 \begin{equation}
E^{1}_{\pm}\left(\dfrac{T}{2n}\right)\=\Set{u \in W^{1,2}\left(\left [0,\left(\dfrac{T}{n}\right)\right],
\C^m\right)| u\left(\dfrac{T}{2n}\right)\in V_{\pm}(N)}\subset E^1\left(\dfrac{T}{2n}\right)
\end{equation}
and we have the decomposition
 \begin{equation}
E^{1}\left(\dfrac{T}{n}\right)\=W^{1,2}\left(\left [0,\left(\dfrac{T}{n}\right)\right],\C^m\right)=E^1_+\left(\dfrac{T}{2n}\right)\oplus E^1_-\left(\dfrac{T}{2n}\right),\end{equation}
then the index form $\mathcal{I}$ is positive semi-definite on $E^1(T/n)$. We let
\begin{equation}
E^{1}\left(\dfrac{T}{n},\omega\right)\=\Set{u \in W^{1,2}\left(\left [0,\left(\dfrac{T}{n}\right)\right],\C^m\right)|u(0)=\omega u(T/n)}\subset E^1\left(\dfrac{T}{n}\right),
\end{equation}
and we observe that, precisely as before, the index form $\mathcal{I}$ is positive semi-definite 
 on $E^{1}\left(\dfrac{T}{n},\omega\right)$ for every $\omega \in \U$. Equivalently, $\mathcal A$ is 
 positie semi-definite on $E^{2}\left(\dfrac{T}{2n},\omega\right)$ for every $\omega \in \U$.

We also observe that the (closed) subspace $E^{1}\left(\dfrac{T}{n},\omega\right)$ can be decomposed as follows
\[
E^{2}(T, \omega)= \bigoplus_{z^m=\omega} E^{2}\left(\dfrac{T}{n},z\right)
\]
and by the same arguments as before, we infer that $\mathcal A$ is semi-positive on $E^{2}(T, \omega)$ for every $\omega \in \U$.
This implies that $\iMor(\omega) \equiv 0$ for every $\omega \in \U$. By invoking
\cite[pag. 172, Theorem 4]{Lon02} it holds that
\[
\iMor(\omega)=i_\omega\big(\gamma(t); t \in [0,T]\big)
\]
and by this equality we infer that $i_\omega\big(\gamma(t); t \in [0,T])\equiv 0$ for every $\omega \in \U$. By the definition of
splitting numbers, it readily follows that
\[
\Big(S^+_{M_T}, S^-_{M_T}\Big)=(0,0)\qquad \textrm{ for every } \omega \in \U.
\]
This conclude the proof.
\end{proof}
\paragraph{Proof of Theorem  \ref{thm:1-ultimasezione-intro} .}
The proof of this result follows directly from Theorem \ref{thm:1-ultimasezione} by invoking \cite[pag.11, Theorem 10]{Eke90}.
(Cf. \cite{HPY17-geo} for further details). This conclude the proof. \qed
\paragraph{Proof of Corollary  \ref{thm:2-ultimasezione-intro} .}
We start to observe that $\omega \in \U$ is an eigenvalue of $M_T$ if and only if $\mathcal A$ is degenerate on $E^2(T,\omega)$. Arguing
precisely as in the proof of Theorem \ref{thm:1-ultimasezione} we infer that since $\mathcal A$ is positive definite on $E^2(T,\omega)$ for every
$\omega \in \U$ then $M_T$ has no eigenvalues on $\U$. Otherwise stated $M_T$ is hyperbolic. This conclude the proof.\qed
\paragraph{Proof of Corollary  \ref{thm:Offin} .}
Since $\mathcal L$ is reversible and $Q$ is $T$-periodic by assumption, we infer that
\[
Q(t)=Q(T-t) 	\qquad t \in [0,T].
\]
Now, let us consider the index form
\[
\mathcal I: E \to \R \textrm{ defined by }
\mathcal I(u,v)= \int_0^T \langle \dot u, \dot v \rangle + \langle Q(t) u, v \rangle \, dt
\]
where $E\=\Set{u \in W^{1,2}\big([0,T], \R^m\big)| u(0)=u(T)}$.
As already observed $u \in \ker \mathcal I$ if and only if $u$ is a $T$-periodic solution of $\mathcal Lu=0$. Now, let us consider the $\Z_2$-action on $E$ defined by
\[
g:E \to E \textrm{ defined by } (g u)(t)\= u(T-t) \qquad t \in [0,T]
\]
(corresponding to consider $N=\Id_m$ in Theorem \ref{thm:1-ultimasezione}). Thus we have
\begin{multline}
E^2_+(T/2)\=\Set{u \in W^{2,2}([0,T/2], \R^m)| \dot{u}(0) =\dot{u}(T/2)=0},\\
E^2_-(T/2)=\Set{u \in W^{2,2}([0,T/2], \R^m)|\ u(T/2)=0,\ \dot{u}(0)=0}.
\end{multline}
By assumption that the Morse index  is $0$   and there are no symmetric $T$-periodic solutions; thus $\mathcal L$ is non-degenerate, too. Then it
readily follows that $\mathcal L$ is positive definite on $E^2_+(T/2)$. The thesis readily follows by invoking Corollary \ref{thm:2-ultimasezione-intro}. This  conclude the proof.\qed

\appendix

\section{The Cyclic and Dihedral group Algebras}\label{subsec:group-algebras}
Let $G$ be a finite group with $n$ elements, namely $\ord(G)=n$. We denote by
$\C[G]$ the
{\em group algebra of $G$ over the complex field\/}, namely the set of all
formal complex linear
combinations of element of $G$ with coefficients in the complex field. It is
well-know that $\C[G]$ is
a $\C$-vector space of dimension $\ord(G)$ and it is a non-commutative algebra
unless the group
is commutative. In what follows we are interested in $\C[G]$-modules. An
important example of $\C[G]$-module
is $\C[G]$ itself viewed as a $\C[G]$-module and it is termed {\em regular
$\C[G]$-module\/}.  As a direct
consequence of the Maschke's Theorem if $V$ is a  $\C[G]$-module then it is
semisimple; thus
there exist simple $\C[G]$-modules $U_1, \dots ,U_k$ such that
\[
 V= U_1 \oplus \dots \oplus U_k.
\]
In particular $\C[G]$  is semisimple. Moreover by the Schur's Lemma, if $V, W$
are simple $\C[G]$-modules either
$\phi:V \to W$ is a $\C[G]$-isomorphism or $\phi$ is the trivial homomorphism.
In the first case, $\phi$ is a
scalar multiple of the identity map $\Id_V$. As direct consequence of Schur's
Lemma, if $G$ is a
finite abelian group, every simple $\C[G]$-module $V$ is of length one. For, let
$x,g \in G$ and $v \in  V$;
thus we have
\[
 x  g\, v = g \, (xv) \qquad g, x \in G, \ \ v \in V.
\]
Therefore, $V \ni v \mapsto x\, v \in V$ is a $\C[G]$-homomorphism. Since $V$ is
 irreducible (being simple), Schur's
Lemma implies that there exists $\lambda_x \in \C$ such that $xv= \lambda_x v$,
for all $v \in V$. In particular,
this implies that every subspace of $V$ is a $\C[G]$-module and since $V$ is
simple, $\dim V=1$. (Actually, also the
converse is true). As a direct consequence of the Wedderburn-Artin
Theorem, we have that $\C[G]$ is isomorphic to the direct sum of matrix algebras
over division rings (since we are
working over an algebraic closed field, it turns out that each division rings is
actually isomorphic
to $\C$), there exists an isomorphism
\[
 \Phi: \C[G] \longrightarrow \Mat(n_1; \C) \oplus \dots \oplus \Mat( n_k; \C),
\]
such that $\ord(G)=\sum_{j=1}^k n_j^2$.
\begin{ex}\label{ex:ciclico}({\bf The cyclic group\/})
 Let $\cic{n}\=\langle r| r^n=1\rangle$ be the cyclic group of order $n$.  It
can be realized as the group
of rotations through angles $2k\pi/n$ around an axis.
Denoting by $\C[z]$ the ring of complex
polynomials, it readily follows that the group algebra $\C[\cic{n}]$ is
isomorphic to the quotient of $\C[z]$ by the
ideal generated by the $n$-th cyclotomic polynomial $z^n-1$; i.e.
\[
 \C[\cic{n}]\cong \C[z]/(z^n-1).
\]
Since we may factor the polynomial $z^n-1$ over the complex field as
$z^n-1=\prod_{k=0}^{n-1} (z- \zeta_n^k)$, where $\zeta_n$ denotes a $n$-th root
of unit,
the group algebra splits as a product of $n$ copies of $\C$ (exactly the
Wedderburn-Artin decomposition), e.g.
\[
 \C[\cic{n}]= \prod_{k=0}^{n-1} \C[z]/(z-\zeta_n^k).
\]
Thus $\C[\cic{n}]$ decomposes as the direct sum of $n$ one-dimensional (complex)
subspaces, i.e.
\[\C[\cic{n}]= \bigoplus_{0 \leq k <n} L_k\]where $\dim_\C L_j=1$ for each $j=1,
\dots,n$.
The generator $t$ acts, in the corresponding factor, as multiplication by
$\zeta_n^k$, meaning that $L_k$ is a
$\cic{n}$-module with respect to the following action $ \cic{n} \times L_k \to
L_k: (t, w)\mapsto \zeta_n^k\, w$ and
the group algebra $\C[\cic{n}]$ splits into the sum of $n$,
$\cic{n}$-modules.
\end{ex}
\begin{rem}
It is worth noticing that, by Example \ref{ex:ciclico} and by invoking the
structure
theorem of finite abelian groups, it
readily follows the group algebra's decomposition of any abelian group.
\end{rem}
\begin{ex}\label{ex:diedrale}({\bf The Dihedral group\/})
Let  $\die{n}$ be the {\em dihedral  group\/} of degree $n$ and
order $2n$. For $n \geq 3$, $\die{n}$ is  the
group of symmetries of a regular $n$-gon in the plane, namely, the group of all
the
plane symmetries that preserves a regular $n$-gon. It contains $n$ rotations,
which form a
subgroup isomorphic to $\cic{n}$ and $n$ reflections.
From an algebraic viewpoint it is
a metabelian  group having the cyclic normal subgroup $\cic{n}$ of index $2$ and
the
following presentation
\begin{equation}\label{eq:diedrale}
 \die{n}\= \langle r,s|r^n=s^2=1,\  srs=r^{-1}\rangle.
\end{equation}
In conclusion $\die{n}$ is a semi-direct product of $\cic{n}$ and $\cic{2}$; in
symbols
$\die{n}= \cic{n} \rtimes \cic{2}$,
where $\cic{n}\=\langle r| r^n=e\rangle$ and $\cic{2}\=\langle s| s^2=e\rangle$.
\footnote{%
We observe that
$\die{1}$ and $\die{2}$ are atypical groups and more precisely $\die{1}$ is the
cyclic group of order $2$ whilst $\die{2}$ is the {\em Klein four-group\/}. It
is worth
noticing that for $n \geq 3$, $\die{n}$ is non-abelian.}
 Each element of $\die{n}$ can  be uniquely written, either in the form $r^k$,
with $0 \leq k \leq n-1$ (if it belongs to $\cic{n}$), or in the form $sr^k$,
with $0 \leq k \leq n-1$  (if
it doesn't belong to $\cic{n}$). Observe that the relation $srs=r^{-1}$ implies
that $sr^ks=r^{-k}$ and
$(sr^k)^2=1$.
By the Wedderburn-Artin  decomposition Theorem, the complex dihedral group
algebra $\C[\die{n}]$ splits as a product of complex matrices as follows
\[
 \C[\die{n}]\cong\bigoplus_{d|n} M_d
\]
where $M_d \cong \C\oplus \C$ if $d=1,2$ and $M_d\cong
\Mat(2; \C)$ if $d >2$.
Denoting by $L_k$ the one-dimensional complex subspace defined in Example
\ref{ex:ciclico}, we
set $\widetilde L_k\=L_k \oplus L_{-k}$ and we identify $L_k$ with the
horizontal subspace $L_k \times \{0\} \subset
\widetilde L_k$ and $L_{-k}$ with the vertical, namely with $\{0\} \times L_{-k}
\subset \widetilde L_k$.  With the action $\phi:\die{n} \times
\widetilde L_k \to \widetilde L_k$ defined by
\begin{equation}
\phi\big(r,(x,y)\big)= \big(\zeta_n\, x, \zeta_n^{-1}\, y\big),
              \qquad
\phi\big(s,(x,y)\big)= (y,x)
\end{equation}
$\widetilde L_k$ is a  $\die{n}$-module.
As long as $k \not \equiv-k$ mod $n$, the module $\widetilde L_k$ is
irreducible. On the
contrary, if $n=2k$, then $\widetilde L_k= L_k \oplus L_k$ decomposes.
Indeed, let $v_+\equiv(1,1)$
and $v_-\equiv(1,-1)$, then the
subspaces $L_+\=\C v_+$ and $L_-\= \C v_-$ are invariant, so $\widetilde L_k$
decomposes as $L_+ \oplus L_-$. On
$L_+$, the element $s$ acts as the identity and $r$ as the
multiplication by $-1$, and on $L_{-}$ both
act as multiplication by $-1$.
Thus the group algebra decomposes in a direct sum
of the $\die{n}$-modules as follows
\[
 \C[\die{n}]= \bigoplus_{k=1}^{n_*}2\, \widetilde L_k\oplus \widetilde L_0
\oplus \delta \,
 \widetilde L_{n/2}
\]
where $\delta \=0$ if $n$ is odd and $1$ if $n$ is even.
\end{ex}
Let $E$ be a separable complex Hilbert space and let $\Id_E$ be denote the
identity operator in $E$.
Given the unitary operator $\mathcal R \in \U(E)$, let $\mathcal C_n$ be the
finite
subgroup defined by $\mathcal C_n\= \langle \mathcal R \in \U(E)| \mathcal R^n =
\Id_E\rangle\subset \U(E).$
(Thus $\mathcal C_n$ is isomorphic to the cyclic group $\cic{n}$ and it
is nothing but a unitary representation of the cyclic group $\cic{n}$).
By using the spectral mapping theorem it readily follows that the spectrum of
$\mathcal R$ is
respectively given by $\sigma(\mathcal R)=\Set{\zeta_n^k\in \C|k=0, \dots,
n-1}.$
Let $E_k\=\ker(\mathcal R-\zeta_n^k\Id_E)$ be the eigenspace corresponding to
the eigenvalue $\zeta_n^k$. By the spectral theory of  normal operators, it is
well-known that
$E_k$ are mutually orthogonal and, clearly $\cic{n}$ invariant, the action
being
given the multiplication by $\zeta_n^k$
\[
 \cic{n} \times E_k \to E_k:(\zeta_n^{k}, v)\mapsto \mathcal
R\,v=\zeta_n^{k}\,v.
\]
Thus, we get a orthogonal decomposition of the Hilbert space $E$ into
$\cic{n}$-closed stable subspaces as follows
\begin{equation}\label{eq:decomposition}
 E=E_1 \oplus \dots \oplus E_n.
\end{equation}
\begin{rem}
We observe  that the  decomposition
given in Equation \eqref{eq:decomposition} is the isotypic decomposition induced
by the unitary
representation of the cyclic group. In particular each subspace $E_k$ is
given by the direct sum of (infinitely many) one-dimensional irreducible
representations
described in Example  \ref{ex:ciclico}.
\end{rem}
We denote by $\mathcal D_n$ the finite subgroup (actually  a unitary
representation of the dihedral group
$\die{n}$) of the unitary group of $E$, presented by
\begin{equation}\label{eq:image-diedrale-unitaria-app}
\mathcal D_n \= \langle\mathcal R, \mathcal S \in \U(E)| \mathcal R^n=
\mathcal S^2=(\mathcal S \mathcal R)^2= \Id_E\rangle\subset \U(E).
\end{equation}
For $k=0, \dots, n-1$, we define the closed linear subspaces
$E_{k}\=\ker(\mathcal R-\zeta_n^k\Id_E)$ and
$ E_{-k}\=\ker(\mathcal R-\zeta_n^{-k}\Id_E)$
and for $k=1,\dots,n_*$, we let
\begin{equation}\label{eq:F_k-app}
F_k\=E_k \oplus E_{-k}.
\end{equation}
We observe that  $F_k$ defined in Equation \eqref{eq:F_k-app} is a
$\die{n}$-module with the action given by
\[
 \cic{n} \times F_k \to F_k:\big(\zeta_n^k,v\big)\longmapsto
 \mathcal R\, v=\begin{bmatrix}
\zeta_n^{k}\, \Id_E&0\\0 & \zeta_n^{-k}\,\Id_E
\end{bmatrix}\,v\quad \textrm { and }
\]
\[
 \cic{2} \times F_k \to F_k:\big(s,v\big)\longmapsto
 \mathcal S\, v=\begin{bmatrix}
0&\Id_E\\ \Id_E &0
\end{bmatrix}\,v.
\]
Thus, we get a decomposition of the Hilbert space $E$ into mutually orthogonal
$\die{n}$-stable modules, given by
\begin{equation}\label{eq:decomposition-3}
 E=\bigoplus_{k=0}^{\bar n}
            F_k.
\end{equation}
\begin{rem}
We observe that the decomposition given in Equation
\eqref{eq:decomposition-3}, is the isotypic decomposition of $E$ with
respect to the irreducible representations of $\die{n}$.
\end{rem}

\section{Maslov index, Spectral Flow and Index Theorems}
\label{sec:Maslov+sf}

The goal of this Section is to briefly recall the Definitions and the main
Properties of the {\em Maslov index\/}
(and its friends) and the Spectral flow. In order to make the presentation as
smooth as possible, we
split this Section into two different Subsections. In the first Subsection
\ref{subsec:Maslov} we start to
recalling  the differentiable structure of the Lagrangian Grassmannian of a real
and complex symplectic space as well as
some well-known facts about the {\em (relative) $L_0$-Maslov index\/} or
more generally the {\em Maslov index for paths  of Lagrangian subspaces\/} both
on real
and complex symplectic spaces.  In order to make the Notations and Definitions
as  uniform as possible we start by
recalling the differentiable  structure of the Lagrangian Grassmannian
 which plays a crucial role
in the description of the
Maslov index through an intersection theory.
Our basic references for the material
contained in the first Subsection are \cite{Arn67, RS93, CLM94, LZ00b, GPP03,
HS09} and references therein.
In Subsection \ref{subsec:spectral-flow}, starting on  some useful functional
analytic preliminaries and
we continue to provide the construction as well as useful properties of
the spectral flow for paths of closed selfadjoint Fredholm operators which are
continuous
in the gap topology, which we'll need in the later Sections. Our basic
references are
\cite{AS69, APS76, RS95, BLP05, Les05, LZ00a, HP17}.


\subsection{The geometry of the Lagrangian Grassmannian and the Maslov
index}\label{subsec:Maslov}

Let $(V,\omega)$ be (finite)  $2m$-dimensional real {\em symplectic vector
space\/}, $\Lin(V)$ be denote
the vector space of all bounded and linear  operators of $V$ and let $J \in
\Lin(V)$ be a
complex structure  {\em compatible\/} with the symplectic form $\omega$, meaning
that  $\omega(\cdot, J\cdot)$
is an inner product on $V$. Let us consider the  {\em Lagrangian Grassmannian of
$(V,\omega)$\/}, namely the set
$\Lagr(V, \omega)$ and we recall that it is a real compact and connected
analytic embedded $m(m+1)/2$-dimensional  submanifold of the Grassmannian
manifold of $V$. Moreover for any
$L \in \Lambda(V,\omega)$ the tangent space $T_L\Lambda(V,\omega)$
is canonically isomorphic to the space  $\Bsym(L)$ of all symmetric bilinear
forms on $L$.
Given  $L_0\in\Lagr(V,\omega)$ and any non-negative integer $j \in \Set{0,\dots,
m}$, we define the sets
$\Lagr^j(L_0;V)\=\big\{L\in\Lagr(V,\omega):\dim(L\cap L_0)=k\big\}$ and we
observe that
$\Lagr(V,\omega)\= \bigcup_{j=0}^m \Lagr^j(L_0;V)$. It is well-known that
$\Lagr^j(L_0;V)$ is a connected embedded analytic  submanifold of
$\Lagr(V,\omega)$ (being  locally
closed  orbit of the Lie group $\Sp(V,\omega, L_0)$ of all symplectomorphism of
$(V, \omega)$ which preserve $L_0$  (cfr. \cite[Theorem 2.9.7]{Var74}),  and
connected) having
codimension equal to $ j(j+1)/2$. In particular $\Lagr^1(L_0;V)$
has codimension $1$ and for $ k \geq 2$ the codimension of $\Lagr^j(L_0;V)$ in
$\Lagr(V,\omega)$
is bigger or equal to $3$.  Its tangent space  is  canonically isomorphic to the
space of
all symmetric bilinear forms over $L$ vanishing on $L \cap L_0$.
A central object in our discussion is played by the
{\em universal Maslov (singular) cycle with vertex at $L_0$\/},
being an algebraic (actually a determinantal) variety defined by
\begin{equation}\label{eq:univ-Maslov-cycle}
\Sigma(L_0;V) \= \bigcup_{j=1}^m \Lagr^j(L_0;V).
\end{equation}
We observe that the Maslov  cycle is the  (topological) closure of lowest
codimensional stratum
$\overline{\Lagr^1(L_0;V)}$. In particular, $\Lagr^0(L_0;V)$, the set of all
Lagrangian
subspaces that are transversal to $L_0$, is an open and  dense subset of
$\Lagr(V,\omega)$.
The (top stratum) codimensional 1-submanifold $\Lagr^1(L_0;V)$ in
$\Lagr(V,\omega)$ is {\em co-oriented\/} or
otherwise stated it carries a {\em transverse orientation\/}. In fact
given $\varepsilon >0$, for each $L \in \Lagr^1(L_0;V)$, the smooth
path of Lagrangian subspaces $\ell:(-\varepsilon, \varepsilon) \to
\Lagr(V,\omega)$ defined by
$\ell(t)\=\exp(tJ)$ crosses $\Lagr^1(L_0;V)$ transversally. The desired
transverse orientation is given by
the direction along the path when the parameter runs between $(-\varepsilon,
\varepsilon)$.
In an equivalent way, the {\em co-orientation or transverse orientation\/}
is meant in the following sense; we first observe that the mapping
\[
 T_L\Lagr(V,\omega) \simeq \Bsym(L) \ni B \mapsto B\vert_{(L_0\cap
L)\times(L_0\cap L)} \in \Bsym(L_0\cap L)
\]
passes to the quotient $ T_L\Lagr(V,\omega)/T_L\Lagr^j(L_0;V) \to \Bsym(L_0\cap
L).$
The hypersurface $\Lagr^1(L_0;V)$ carries a {\em canonical transverse
orientation\/} which is defined by
declaring that a vector $B\in T_L\Lagr(V,\omega), B \notin T_L\Lagr^1(L_0;V)$ is
{\em positively oriented\/}
if the non-zero symmetric bilinear form $B\vert_{(L\cap L_0)\times (L\cap L_0)}$
on the line
$L\cap L_0$ is positive definite. Thus the Maslov cycle is two-sidedly embedded
in $\Lagr(V,\omega)$. Based on
these properties, Arnol'd in \cite{Arn67}, defined an intersection index for
closed loops in $\Lagr(V,\omega)$
(actually in $(\R^{2m}, \omega)$, but the treatment in this more general
situation presents no difficulties) via
transversality arguments. This general position arguments can be generalised to
Lagrangian paths (not only closed)
with endpoints out of the Maslov cycle. Following authors in \cite{CLM94, HS09}
we introduce the following Definition.
\begin{defn}\label{def:Maslov-index}
Let $L_0 \in \Lagr(V,\omega)$ and, for $a<b$, let  $\ell\in \mathscr
C^0\big([a,b],\Lagr(V,\omega)\big)$. We define
the {\em (relative) Maslov index of $\ell$ with respect to $L_0$\/} as the
integer given by
\begin{equation}\label{eq:intersection-number}
 \iCLM(L_0, \ell)\=\left[\exp(-\varepsilon J)\,\ell: \Sigma(L_0;V)\right]
\end{equation}
where $\varepsilon \in (0,1)$ is sufficiently small and where the right-hand
side denotes the intersection number.
\end{defn}
\begin{rem}
A few Remarks on the Definition  \ref{def:Maslov-index} are in order. By the
basic geometric
observation given in \cite[Lemma 2.1]{CLM94}, it readily follows that there
exists $\varepsilon >0$
sufficiently small such that $\exp(-\varepsilon J)\,\ell(a), \exp(-\varepsilon
J)\,\ell(b)$ doesn't lie
on $\Sigma(L_0;V)$. By \cite[Step 2, Proof of Theorem 2.3]{RS93}, there exists a
perturbed path $\widetilde \ell$
having only {\em simple crossings\/} (namely the path $\ell$ intersects the
Maslov cycle transversally and  in
the top stratum. Since, simple crossings are isolated, on a compact interval are
in a finite number. To each
crossing instant  $t_i \in (a,b)$ we associate the number $s(t_i)= 1$ (resp.
$s(t_i)=-1$) according to the fact that,
in a sufficiently small neighbourhood of $t_i$, $\widetilde \ell$ have the same
(resp. opposite) direction of
$\exp(t\,J) \widetilde \ell(t_i)$. Then the intersection number given in Formula
\eqref{eq:intersection-number} is
equal to the summation of $s(t_i)$, where the sum runs over all crossing
instants $s(t_i)$.
\end{rem}
\noindent
The Maslov index given in Definition \ref{def:Maslov-index} have many important
properties (cfr. \cite{RS93, CLM94}
for further dails). Below we list only a few of them that we'll use in the
sequel for computing the Maslov index and,
for further details, we refer the interested reader to the aforementioned paper
and references therein.
\begin{enumerate}
 \item[]{\bf Property I (Reparametrisation Invariance)\/}
 Let $\psi:[a,b] \to [c,d]$ be a  continuous function with $\psi(a)=c$ and
$\psi(b)= d$.
 Then $ \iCLM(L_0, \ell)=\iCLM(L_0, \ell\circ \psi).$
 \item[]{\bf Property II (Homotopy invariance Relative to the Ends)\/} Let
\[
\overline  \ell: [0,1] \times [a,b] \to
\Lagr(V,\omega):(s,t)\mapsto \overline\ell(s,t)
\]
be a continuous two-parameter family of Lagrangian subspaces such that
$\dim\big(L_0\cap \overline \ell(s,a)\big) $
and $\dim\big(L_0\cap \overline{\ell}(s,b)\big) $ are independent on $s$. Then $
 \iCLM(L_0, \overline\ell_0) = \iCLM(L_0, \overline\ell_1)$
where $\overline \ell_0(\cdot)\= \overline\ell(0, \cdot)$ and
$\overline \ell_1(\cdot)\= \overline\ell(1, \cdot)$.
\item[]{\bf Property III (Path Additivity)\/} If $c \in (a,b)$, then
\[
\iCLM(L_0, \ell)= \iCLM(L_0, \overline\ell\vert_{[a,c]})+ \iCLM(L_0,
\overline\ell\vert_{[c,b]}).
\]
\item[]{\bf Property IV (Symplectic Invariance)\/} Let $\phi \in \mathscr
C^0\big([a,b], \Sp(V,\omega)\big)$ be a
 continuous path in the (closed) symplectic group $\Sp(V,\omega)$ of all
symplectomorphisms of $(V,\omega)$. Then
 \[
  \iCLM(L_0, \ell)= \iCLM\big(\phi(t)\, L_0, \phi(t)\, \ell(t)\big), \qquad t
\in [0,1].
 \]
 \item[]{\bf Property V (Symplectic Additivity)\/} For $i=1,2$ let
 $(V_i, \omega_i)$  be symplectic  vector spaces, $L_i \in \Lagr(V_i,\omega_i)$
and let $\ell_i \in
 \mathscr C^0\big([a,b], \Lagr(V_i,\omega_i\big)$. Then
 \[
  \iCLM(\ell_1\oplus \ell_2, L_1\oplus L_2)= \iCLM(\ell_1, L_1)+ \iCLM(\ell_2,
L_2).
 \]
 \end{enumerate}
Although the Definition of the Maslov index given in Formula
\ref{eq:intersection-number} is apparently
simple, the computation of this homotopy invariant is, in general, quite
involved. One  efficient technique
for computing this invariant, was introduced (in the non-degenerate case)
by the authors in \cite{RS93} through  the so-called {\em crossing forms\/} and,
generalised (in the degenerate
situation) by authors in \cite{GPP03, GPP04}.
For $\varepsilon >0$ let $\ell^* :(-\varepsilon, \varepsilon) \to
\Lagr(V,\omega)$ be a $\mathscr C^1$-path such that
$\ell^*(0)=L$. Let $L_1$ be a fixed Lagrangian complement of
$L$ and, for $v \in L$ and for sufficiently small $t$ we define $w(t) \in L_1$
such that
$v+w(t) \in \ell^*(t)$. Then the  form
\begin{equation}\label{eq:laq}
 Q[v]=\dfrac{d}{dt}\Big|_{t=0} \omega\big(v, w(t)\big)
\end{equation}
is independent of the choice of $L_1$. A {\em crossing instant\/} $t_0$ for the
continuous curve $\ell:[a,b] \to \Lagr(V,\omega)$ is an instant such that
$\ell(t_0)\in \Sigma(L_0;V)$. If the
curve is $\mathscr C^1$, at each
crossing, we define the {\em crossing form\/} as  the quadratic form on
$\ell(t_0)\cap L_0$ given by
\[
\Gamma(\ell, L_0, t_0)= Q(\ell(t_0, \dot \ell(t_0)\Big\vert_{\ell(t_0)\cap L_0}
\]
where $Q$ was defined in Formula \eqref{eq:laq}. A crossing $t_0$ is called {\em
regular\/} if the crossing
form is non-degenerate; moreover if the curve $\ell$ has only
regular crossings we shall refer as a {\em regular path. \/}(Heuristically,
$\ell$ has only regular
crossings if and only if it is transverse to $\Sigma(L_0)$).  Following authors
in \cite{LZ00b}, if
$\ell: [a,b] \to \Lagr(V,\omega)$ is a regular $\mathscr C^1$-path,
then the crossing instants are in a finite number and the Maslov index is given
by:
\[
 \iCLM(L_0,\ell)=\coindex\left[{\Gamma(\ell(a), L_0, a)}\right]+
 \sum_{\substack{t_0 \in \ell^{-1}\big(\Sigma(L_0;V)\big)\\t_0 \in
]a,b[}}\ssgn{\Gamma(\ell(t_0), L_0, t_0)} -
 \iMor{\left[\Gamma(\ell, L_0, b)\right]},
\]
where $\coindex, \iMor $ denotes respectively the number of positive ({\em
coindex\/}), negative eigenvalues ({\em
index\/}) in the  Sylvester's Inertia Theorem and where $\sgn\=\coindex-\iMor$
denotes the ({\em signature\/}). We
observe that any $\mathscr C^1$-path is homotopic through a fixed endpoints
homotopy to a path
having  only regular  crossings.

In order to prove a dihedral equivariant Bott-type iteration formula by using
the  Maslov-type index, we need to review
the Maslov index theory of the complex Lagrangian subspaces. For, let $(\C^{2m},
\omega)$ be the complex symplectic
vector space with the symplectic form $\omega(x,y)= \langle J x, y\rangle, $ for
all $ x,y \in \C^{2m},$
where  $\langle \cdot, \cdot \rangle$ denotes the standard Hermitian product in
$\C^{2m}$,
$ J=\begin{bmatrix}
           0 & -\Id_m\\ \Id_m & 0
          \end{bmatrix}$
and $\Id_m$ is the identity matrix. If no confusion is possible we'll omit the
subindex $m$. We recall that a {\em
complex subspace $L$ is Lagrangian\/} if $\omega|_L =0$ and the $\dim_\C L=m$.
We let $L^\pm = \ker (iJ \mp \I_{2m})$
and we observe that $L$ is a (complex) Lagrangian subspace $L$ if and only if it
can be seen as the graph of a unitary operator $U: L^+\to L^-$. Otherwise stated
the {\em complex Lagrangian Grassmannian\/}
is homeomorphic  to the unitary group of $\C^{2m}$:
\begin{equation}\label{eq:omeomorfismo}
  \Lagr(\C^{2m},\omega) \cong \U(m).
\end{equation}
(Cfr.  \cite{Edw64, Arn00, Zhu06, Por10} and references therein, for further
details).
We denote by $\mathbb F$ the homeomorphism defined in Formula
\eqref{eq:omeomorfismo} and we observe that
\[
\dim \big(L_1 \cap L_2\big) = \dim \ker \big(\mathbb F(L_2)^{-1} \mathbb F(L_1)
-\Id_m\big) .
\]
For any fixed $U \in \U(m)$, we let $\Sigma(U)\=\Set{U_0 \in \U(m)|
\det\big(U^{-1}\, U_0  -\Id_m\big)=0}.$
We refer to $\Sigma(U)$ as the {\em singular cycle of $U$.\/} We observe that,
for any $U_0 \in \Sigma(U)$, there
exists $\varepsilon >0$ sufficiently small such that  $e^{it} U_0$  is
transversal to $\Sigma(U)$. Let now
$\mathcal U:[a,b] \to \U(m)$ be a continuous path.
For $\varepsilon >0$ small enough, $e^{-\varepsilon \, i} \, \mathcal U(a)$ and
$e^{-\varepsilon \, i} \, \mathcal U(b)$ are out of the singular cycle of $U$
and the intersection number of the
perturbed path $e^{\varepsilon\,i}\mathcal U$  with the singular cycle
$\Sigma(U)$, is well-defined. For this
reason we are entitled to introduce the following Definition.
\begin{defn}\label{def:Maslov-complex}
 Let $L \in \Lagr(\C^{2m}, \omega)$ be fixed and let $\ell:[a,b]
\to\Lagr(\C^{2m}, \omega)$ be a continuous
 path. We define the {\em (complex) Maslov index\/} as $ \iCLM(L, \ell;[a,b])\=
[e^{-\varepsilon \, i}\mathbb F(\ell):
 \Sigma(\mathbb F(L)].$
 \end{defn}
 \noindent
Given  $L \in \Lagr(\R^{2m}, \omega)$ be a real Lagrangian subspace, then
$L^\C\= L \otimes \C \in \Lagr(C^{2m}, \omega)$.
We define the manifold $\Lagr_\R(\C^{2m}, \omega)\=\Set{L= \overline L | L \in
\Lagr(C^{2m},\omega)}$
and we observe that $\Lagr_\R(\C^{2m}, \omega)$ is isomorphic to $\Lagr(\R^{2m},
\omega) \otimes \C$. It is worth
noticing that, given $L \in \Lagr(\C^{2m}, \omega)$, we have $\mathbb
F(e^{-\varepsilon J}L)= e^{-2 \varepsilon i}
\mathbb F(L)$. Thus the (real) Maslov index given in Definition
\ref{def:Maslov-index} coincides with the (complex)
Maslov index given in Definition \ref{def:Maslov-complex} when the path consists
of real Lagrangian subspaces. So the
Maslov index of a path of real Lagrangian subspaces is the same as the path of
complex Lagrangian subspaces. We
conclude by observing that the Definition of the crossing form as well as the
computation of the Maslov index for a
$\mathscr C^1$ regular path through crossing forms is the same as in the real
case and it also fulfil
Properties I-V given above.



\subsection{On the spectral flow for paths of closed self-adjoint Fredholm
operators}\label{subsec:spectral-flow}

The aim of this Subsection is to briefly recall the Definition and the main
properties of the spectral
flow for a continuous path of closed self-adjoint Fredholm operator. It is
well-known that this topological
invariant was introduced by authors in \cite{APS76} in order to develop an Index
Theory on manifolds with
boundary. Since then, it has been extensively applied and investigated
extensively.

Let $E$ be a separable complex Hilbert space and let
$T: \mathcal D(T) \subset E \to E$ be  a  self-adjoint
Fredholm
operator. By the Spectral decomposition Theorem (cf., for instance,
\cite[Chapter III,
Theorem 6.17]{Kat80}), there is an orthogonal decomposition $
 E= E_-(T)\oplus E_0(T) \oplus E_+(T),$
that reduces the operator $T$
and has the property that
\[
 \sigma(T) \cap (-\infty,0)=\sigma\big(T_{E_-(T)}\big), \quad
 \sigma(T) \cap \{0\}=\sigma\big(T_{E_0(T)}\big),\quad
 \sigma(T) \cap (0,+\infty)=\sigma\big(T_{E_+(T)}\big).
\]
\begin{defn}\label{def:essential}
Let $T \in \mathcal{CF}^{sa}(E)$. We  term $T$ {\em essentially
positive\/}
if $\sigma_{ess}(T)\subset (0,+\infty)$, {\em essentially negative\/} if
$\sigma_{ess}(T)\subset (-\infty,0)$ and finally
{\em strongly indefinite\/} respectively if $\sigma_{ess}(T) \cap (-\infty,
0)\not=
\emptyset$ and $\sigma_{ess}(T) \cap ( 0,+\infty)\not=
\emptyset$.
\end{defn}
\noindent
If $\dim E_-(T)<\infty$,
we define its {\em Morse index\/}
as the integer denoted by $\iindex{T}$ and defined as $
 \iindex{T} \= \dim E_-(T).$

Given $T \in\cfsa(E)$, for  $a,b \notin
\sigma(T)$ we set $
\mathcal P_{[a,b]}(T)\=\Real\left(\dfrac{1}{2\pi\, i}\int_\gamma
(\lambda-T)^{-1} d\, \lambda\right)$
where $\gamma$ is the circle of radius $\dfrac{b-a}{2}$ around the point
$\dfrac{a+b}{2}$. We recall that if
$[a,b]\subset \sigma(T)$ consists of  isolated eigenvalues of finite type then
\begin{equation}\label{eq:range-proiettor}
 \im \mathcal P_{[a,b]}(T)= E_{[a,b]}(T)\= \bigoplus_{\lambda \in (a,b)}\ker
(\lambda -T);
\end{equation}
(cf. \cite[Section XV.2]{GGK90}, for instance) and $0$ either belongs in the
resolvent set of $T$ or it is an isolated eigenvalue of finite multiplicity.
Let us now consider the {\em graph distance topology\/} which is the topology
induced by the {\em gap
metric\/} $d_G(T_1, T_2)\=\norm{P_1-P_2}$
where $P_i$ is the projection onto the graph of $T_i$ in the product space
$E
\times E$. The next result allow us to  define the spectral flow for
gap
continuous paths in  $\cfsa(E)$.
\begin{prop}\label{thm:cor2.3}
 Let $T_0 \in \cfsa(E)$ be fixed.
 \begin{enumerate}
  \item[(i)] There exists a positive real number $a \notin \sigma(T_0)$ and an
open neighborhood $\mathscr N \subset  \cfsa(E)$ of $T_0$ in the gap
topology
such that $\pm a \notin
\sigma(T)$ for all $T \in  \mathscr N$ and the map
 \[
  \mathscr N \ni T \longmapsto \mathcal P_{[-a,a]}(T) \in \Lin(E)
 \]
is continuous and the projection $\mathcal P_{[-a,a]}(T)$ has constant finite
rank for all $T \in \mathscr N$.
 \item[(ii)] If $\mathscr N$ is a neighborhood as in (i) and $-a \leq c \leq d
\leq a$ are such that $c,d \notin
 \sigma(T)$ for all $T \in \mathscr N$, then $T \mapsto \mathcal P_{[c,d]}(T)$
is
continuous on $\mathscr N$.
 Moreover the rank of $\mathcal P_{[c,d]}(T) \in \mathscr N$ is finite and
constant.
 \end{enumerate}
\end{prop}
\begin{proof}
For the proof of this result we refer the interested reader to
\cite[Proposition 2.10]{BLP05}.
\end{proof}
Let $\mathcal A:[a,b] \to \cfsa(E)$ be a gap continuous path.  As
consequence
of Proposition \ref{thm:cor2.3}, for every $t \in [a,b]$ there exists $a>0$ and
an open
connected neighborhood $\mathscr N_{t,a} \subset \cfsa(E)$ of
$\mathcal
A(t)$
such that $\pm a \notin \sigma(T)$ for all $T \in \mathscr N_{t,a}$ and the map
$\mathscr N_{t,a} \in T \longmapsto \mathcal P_{[-a,a]}(T) \in \mathcal
B$
is continuous and hence $ \rk\left(\mathcal P_{[-a,a]}(T)\right)$ does not
depends on $T \in \mathscr N_{t,a}$. Let us consider the open covering
of the interval $[a,b]$ given by the
pre-images of the neighborhoods $\mathcal
N_{t,a}$ through $\mathcal A$ and, by choosing a sufficiently fine partition of
the interval $[a,b]$ having diameter less than the Lebesgue number
of the covering, we can find  $a=:t_0 < t_1 < \dots < t_n:=b$,
operators $T_i \in \cfsa(E)$ and
positive real numbers $a_i $, $i=1, \dots , n$ in such a way the restriction of
the path $\mathcal A$ on the
interval $[t_{i-1}, t_i]$ lies in the neighborhood $\mathscr N_{t_i, a_i}$ and
hence the
$\dim E_{[-a_i, a_i](\mathcal A_t)}$ is constant for $t \in [t_{i-1},t_i]$,
$i=1, \dots,n$.
\begin{defn}\label{def:spectral-flow-unb}
The \emph{spectral flow of $\mathcal A$} (on the interval $[a,b]$) is defined by
\[
 \spfl(\mathcal A, [a,b])\=\sum_{i=1}^N \dim\,E_{[0,a_i]}(\mathcal A_{t_i})-
 \dim\,E_{[0,a_i]}(\mathcal A_{t_{i-1}}) \in \Z.
\]
\end{defn}
(In shorthand Notation we  denote  $\spfl(\mathcal A, [a,b])$ simply  by
$\spfl(\mathcal A)$ if no confusion  is possible).
The spectral flow as given in Definition \ref{def:spectral-flow-unb} is
well-defined
(in the sense that it is independent either on the partition or on the $a_i$)
and only depends on
the continuous path $\mathcal A$. (Cfr. \cite[Proposition 2.13]{BLP05} and
references therein).
We list some useful properties of the spectral flow and we refer to \cite{BLP05}
for further details.
\begin{itemize}
 \item[]  {\bf Property I (Path Additivity)\/} If $\mathcal A_1: [a,b] \to
 \cfsa(E)$, $\mathcal A_1,\mathcal
A_2: [c,d] \to
 \cfsa(E)$ are two continuous path such that
$\mathcal A_1(b)=\mathcal A_2(c)$, then
 \[
  \spfl(\mathcal A_1 *\mathcal A_2) = \spfl(\mathcal A_1)+\spfl(\mathcal A_2).
 \]
 \item[] {\bf Property II (Homotopy Relative to the Ends)\/} If $
  h: [0,1]\times [a,b] \to \cfsa(E):(s,t)\mapsto h(s,t)$ is a
continuous map such that $
  h_a: [0,1] \ni s \mapsto \dim\ker h(s,a)\in \cfsa(E)$ and
   $
   h_b: [0,1]\ni t \mapsto \dim\ker h(s,b)\in \cfsa(E)$ are independent
on $s$, then
   \[
    \spfl(h_0,[a,b])=\spfl(h_1, [a,b]),
   \]
   where $h_0(\cdot)\= h(0, \cdot)$ and $h_1(\cdot)=h(1,\cdot)$.
   \item[] {\bf Property III (Direct sum)\/} If for $i=1,2$, $E_i$ are
Hilbert spaces and if $h_i:[a,b] \to
   \cfsa(E_i)$ are two gap-continuous paths of self-adjoint Fredholm
operators, the
   \[
    \spfl(h_1\oplus h_2,[a,b])= \spfl(h_1,[a,b]) + \spfl(h_2, [a,b]).
   \]
\end{itemize}
As already observed, the spectral flow, in general,  depends on the whole path
and not
just on the ends. However, if the path has a special form, it actually depends
on the
end-points. More precisely, let  $\mathcal A ,\mathcal B\in \cfsa(E)$
and let $\widetilde{\mathcal A}:[a,b] \to \cfsa(E)$ be the path
pointwise defined by $\widetilde{\mathcal A}(t)\=\mathcal A+ \widetilde{\mathcal
B}(t)$  where $
\widetilde{\mathcal B}$ is any continuous curve of $\mathcal A$-compact
operators parametrised on $[0,1]$
such that  $\widetilde{\mathcal B}(0)\=0$ and  $ \widetilde{\mathcal B}(1)\=
\mathcal B$. In this case,
the spectral flow depends of the
path $\widetilde A$, only on the endpoints (cfr. \cite{LZ00a} and reference
therein).
\begin{rem}
 It is worth noticing that, since every operator $\widetilde{\mathcal A}(t)$ is
a compact perturbation of a
 a fixed one, the path $\widetilde{\mathcal A}$ is actually a continuous path
into $\Lin(\mathcal W; E)$,
 where $\mathcal W\=\mathcal D(\mathcal A)$.
\end{rem}
\begin{defn}\label{def:rel-morse-index}(\cite[Definition 2.8]{LZ00a}).
 Let $\mathcal A ,\mathcal B\in \cfsa(E)$ and we assume that $\mathcal
B$ is
 $\mathcal A$-compact (in the sense specified above). Then the
{\em  relative Morse index of the pair $\mathcal A$, $\mathcal A+\mathcal B$\/}
is
defined by
 \[
  \irel(\mathcal A, \mathcal A+\mathcal B)=-\spfl(\widetilde{\mathcal A};[a,b])
 \]
where $\widetilde{\mathcal A}\=\mathcal A+ \widetilde{\mathcal B}(t)$ and where
$
\widetilde{\mathcal B}$ is any continuous curve parametrised on $[0,1]$
of $\mathcal A$-compact operators such that
$\widetilde{\mathcal B}(0)\=0$ and
$ \widetilde{\mathcal B}(1)\= \mathcal B$.
\end{defn}
\noindent
In the special case in which the Morse index of both operators $\mathcal A$ and
$\mathcal A+\mathcal B$ are
finite, then
\[
\irel(\mathcal A, \mathcal A+\mathcal B)=\iindex{\mathcal A +\mathcal
B}-\iindex{\mathcal A}.
\]
Let  $\mathcal W, E$ be separable Hilbert spaces with a dense and
continuous
inclusion $\mathcal W \hookrightarrow E$ and let
$\mathcal A:[0,1] \to \cfsa(E)$  having fixed domain $\mathcal W$. We
assume that $\mathcal A$ is
a continuously differentiable path  $\mathcal A: [0,1] \to \cfsa(E)$
and
we denote by $\dot{\mathcal A}_{\lambda_0}$ the derivative of
$\mathcal A_\lambda$ with respect to the parameter $\lambda \in [0,1]$ at
$\lambda_0$.
\begin{defn}\label{def:crossing-point}
 An instant $\lambda_0 \in [0,1]$ is called a {\em crossing instant\/} if
$\ker\, \mathcal A_{\lambda_0} \neq 0$. The
 crossing form at $\lambda_0$ is the quadratic form defined by
\begin{equation}
 \Gamma(\mathcal A, \lambda_0): \ker \mathcal A_{\lambda_0} \to \R, \quad
\Gamma(\mathcal A, \lambda_0)[u] =
\langle \dot{\mathcal A}_{\lambda_0}\, u, u\rangle_E.
\end{equation}
Moreover a  crossing $\lambda_0$ is called {\em regular\/}, if $\Gamma(\mathcal
A, \lambda_0)$ is non-degenerate.
\end{defn}
We recall that there exists $\varepsilon >0$ such that   $\mathcal A +\delta \,
\Id_E$ has only regular crossings
  for almost every $\delta \in (-\varepsilon, \varepsilon)$. (Cfr., for instance
\cite[Theorem 2.6]{Wat15}
  and references therein).
In the special case in which all crossings are regular, then the spectral flow
can be easily computed through  the
crossing forms. More precisely the following result  holds.
\begin{prop}
 If $\mathcal A:[0,1] \to \cfsa(\mathcal W, E)$ has only regular
crossings then they are in a finite
 number and
 \[
  \spfl(\mathcal A, [0,1]) = -\iMor{\left[\Gamma(\mathcal A,0)\right]}+
\sum_{t_0 \in (0,1)}
  \sgn\left[\Gamma(\mathcal A, t_0)\right]+
  \coiindex{\Gamma(\mathcal A,1)}
 \]
where the sum runs over all the crossing instants.
\end{prop}
\begin{proof}
 The proof of this result follows by arguing as in \cite{RS95}. This conclude
the proof.
\end{proof}



\newpage
\vspace{1cm}
	\noindent
	\textsc{Prof. Xijun Hu}\\
	Department of Mathematics\\
	Shandong University\\
	Jinan, Shandong, 250100 \\
	The People's Republic of China \\
	China\\
	E-mail: \email{xjhu@sdu.edu.cn}

\vspace{1cm}
\noindent
\textsc{Prof. Alessandro Portaluri}\\
Department of Agriculture, Forest and Food Sciences\\
Università degli Studi di Torino\\
Largo Paolo Braccini 2 \\
10095 Grugliasco, Torino\\
Italy\\
Website: \url{aportaluri.wordpress.com}\\
E-mail: \email{alessandro.portaluri@unito.it}

\vspace{1cm}
\noindent
\textsc{Dr. Ran Yang}\\
Department of Mathematics\\
Shandong University\\
Jinan,Shandong,250100\\
The People's Republic of China \\
China\\
E-mail: \email{yangran201311260@mail.sdu.edu.cn}

\vspace{1cm}
\noindent
COMPAT-ERC Website: \url{https://compaterc.wordpress.com/}\\
COMPAT-ERC Webmaster \& Webdesigner: Arch.  Annalisa Piccolo

\end{document}